\documentclass[leqno,draft,11pt]{amsart}
\input{amssym.def}

\newtheorem{coro}[equation]{Corollary}

\newtheorem{propo}[equation]{Proposition}

\theoremstyle{definition} \theoremstyle{remark}

\numberwithin{equation}{section}

\begin{document}

\title[Area Littlewood-Paley functions in Hermite and Laguerre settings]
{Area Littlewood-Paley functions associated with Hermite and
Laguerre operators}

\begin{abstract}
In this paper we study $L^{p}$-boundedness properties for area
Littlewood-Paley functions associated with heat semigroups for
Hermite and Laguerre operators.
\end{abstract}

\author[J. Betancor]{Jorge J. Betancor}
\address{Departamento de An\'{a}lisis Matem\'{a}tico\\
Universidad de la Laguna\\
Campus de Anchieta, Avda. Astrof\'{\i}sico Francisco S\'{a}nchez, s/n\\
38271 La Laguna (Sta. Cruz de Tenerife), Espa\~na}
\email{jbetanco@ull.es, lrguez@ull.es}

\author[S.M. Molina]{Sandra M. Molina}
\address{Departamento de Matem\'{a}ticas. Facultad de Ciencias Exactas y Naturales\\Universidad Nacional de Mar del Plata\\
Funes 3350 (7600)\\
Mar del Plata, Argentina} \email{smolina@mdp.edu.ar}

\author[L. Rodr\'{\i}guez-Mesa]{Lourdes Rodr\'{\i}guez-Mesa}

\thanks{This paper is partially supported by MTM2007/65609.}

\maketitle

\section{Introduction.}
We denote by $\mathbb{T}=\{T_{t}\}_{t>0}$ a semigroup of linear
and bounded operators on $L^{p}(\Omega,d\mu)$, $1\leq p<\infty$,
where $(\Omega,d\mu)$ is a measure space. Suppose that $\rho$ is a
metric on $\Omega$ and that, for every $f\in L^{p}(\Omega,d\mu)$,
the mapping
$$
\begin{array}{rll}
M_{f}:(0,\infty) & \rightarrow & L^{p}(\Omega,d\mu)\\
\:\:\:\:t & \rightarrow & M_{f}(t)=T_{t}(f),\\
\end{array}
$$
is a.e pointwise differentiable. For every $q>1$, the area g-function
$g_{\mathbb{T}}^q(f)$ of $f\in L^{p}(\Omega,d\mu)$, $1\leq
p<\infty$, is defined by
$$
g_{\mathbb{T}}^q(f)(x)=\Biggl\{\int_{\Gamma(x)}\biggl|\left(s\frac{\partial}{\partial
s}T_{s}(f)(y)\right)_{|s=t^{2}}\biggr|^q\:
\frac{dtdy}{t^{2}}\Biggr\}^{1/q},
$$
where $\Gamma(x)=\{(y,t)\in\Omega\times (0,\infty) : \rho(x,y)<
t\}$, $x\in\Omega$. This area g-function can be seen as an
extension of the Lusin area integral function. As it is wellknown
Lusin area integral is related to the nontangential boundary
behaviour of analytic and harmonic functions in the unit disc
(see, for instance, the celebrated papers \cite{Cal}, \cite{MZ}
 and \cite{Spen}). $L^{p}$-boundedness properties of the (sublinear)
operator $g_{\mathbb{T}}^q$ (and some extensions of this one) when
$\mathbb{T}$ represents Poisson or heat semigroups associated with
the Euclidean Laplacian operator and in other settings have been
studied by several authors (\cite{AtPi}, \cite{BaMoo},
\cite{BuGun}, \cite{Dahl}, \cite{DKPV}, \cite{Gas},
\cite{GuWh}, \cite{HMY}, \cite{KoVa}, \cite{MuUchi}, \cite{Seg},
\cite{Stein1}, amongst others). In this paper we prove
$L^{p}$-boundedness properties for the function $g_{\mathbb{T}}^q$, $q>1$,
when $\mathbb{T}$ is the heat semigroup associated with Hermite
and Laguerre operators.

For every $n\in \mathbb{N}$, we denote by $h_{n}$ the Hermite
function defined by
$$
h_{n}(x)=\frac{1}{\sqrt{\sqrt{\pi}2^{n}n!}}H_{n}(x)e^{-x^{2}/2},\quad
x\in\mathbb{R},
$$
where $H_{n}$ represents the Hermite polynomial of degree $n$. The
sequence $\{h_{n}\}_{n\in\mathbb{N}}$ is complete and orthonormal
in $L^{2}(\mathbb{R})$. Moreover, one has
$$\mathcal{H}h_{n}=\Big(n+\frac{1}{2}\Big)h_{n},\quad n\in\mathbb{N},$$
where $\mathcal{H}=\frac{1}{2}(-\Delta +|x|^{2})$ is the harmonic
oscillator, also called, Hermite operator. This operator $\mathcal{H}$ is
positive and symmetric in $L^{2}(\mathbb{R})$ on the domain
$C_{c}^{\infty}(\mathbb{R})$, the space of the
$C^{\infty}$-functions on $\mathbb{R}$ which have compact support.

The heat diffusion semigroup $\mathbb{W}=\{W_{t}\}_{t>0}$ generated by $L$
is given by
$$ W_{t}f=\sum_{n=0}^{\infty}c_{n}(f) e^{-(n+1/2)t}h_{n}, \quad  f\in
L^{2}(\mathbb{R}),$$ being, for every $n\in\mathbb{N}$,
$$
c_{n}(f)=\int_{-\infty}^{+\infty}h_{n}(x) f(x)\: dx.
$$
By using the Mehler formula (\cite[p. 380]{Sz}) we can write, for
$f\in L^2(\mathbb{R})$,
$$
W_{t}f(x)=\int_{-\infty}^{+\infty} W_{t}(x,y) f(y)\:dy, \quad
x\in \mathbb{R},
$$
where, for each $x,y\in\mathbb{R}$ and $t>0$,
\begin{equation}\label{kernelW}
 W_{t}(x,y)= \sum_{n=0}^{\infty}e^{-(n+1/2)t}h_{n}(x) h_{n}(y)
 =\frac{1}{\sqrt{\pi}}\Bigl(\frac{e^{-t}}{1-e^{-2t}}\Bigr)^{1/2}
e^{-\frac{(x-e^{-t}y)^2+(y-e^{-t}x)^2}{2(1-e^{-2t})}}.
\end{equation}

 The study of harmonic analysis operators (such as maximal operators,
 Riesz transforms, g-functions,...) in the Hermite polynomial
 setting was begun by Muckenhoupt (\cite{Mu1} and \cite{Mu5}). We can remark
 also in this context the results of Sj\"ogren (\cite{Sj2} and
 \cite{Sj1}), P\'erez and Soria (\cite{PS}), P\'erez (\cite{Pe}), Garc\'ia-Cuerva,
 Mauceri,  Sj\"ogren and Torrea (\cite{GMST1}, \cite{GMST2}), Garc\'ia-Cuerva,
 Mauceri,  Meda, Sj\"ogren and Torrea (\cite{GMMST}),
 Fabes, Guti\'errez and Scotto (\cite{FGS}), Urbina (\cite{Ur}) and Harboure,
 Torrea and Viviani (\cite{HTV1}), amongst others.
 Harmonic analysis associated with the Hermite operator $\mathcal{H}$ has
 been developed in the last years. Stempak and Torrea studied
 maximal operators associated with $\mathbb{W}=\{W_{t}\}_{t>0}$, Riesz
 transforms and certain Littlewood-Paley
 g-functions in \cite{StTo1}, \cite{StTo2} and \cite{StTo3}. Here we consider, for every $q>1$,  the area g-function $g_\mathbb{W}^q$ associated with the heat semigroup
 $\mathbb{W}=\{W_{t}\}_{t>0}$, defined by
 $$
 g_{\mathbb{W}}^q(f)(x)=\Biggl \{\int_{\Gamma(x)}\Bigl|\left(s\frac{\partial}{\partial s}W_{s}(f)(y)\right)
 _{|s=t^{2}}\Bigr|^q\:\frac{dtdy}{t^{2}}\Biggr\}^{1/q},\quad x\in \mathbb{R} ,
 $$
where $\Gamma(x)=\{(y,t)\in \mathbb{R}\times (0,\infty):
|x-y|<t\}$, for every $x\in\mathbb{R}$. We establish the following
$L^{p}$-boundedness properties for this g-function.

\begin{propo}\label{Hermite} Let $q\geq 2$. Then $g_{\mathbb{W}}^q$ defines a
bounded operator from $L^{p}(\mathbb{R})$ into itself, for
every $1<p<\infty$, and from $L^{1}(\mathbb{R})$ into
$L^{1,\infty}(\mathbb{R})$.
\end{propo}

We remark that \cite[Theorem 4.8]{HVP} can not be used to establish the $L^p$-boundedness of the operator defined by $g_{\mathbb{W}}^2$ because
$$
\int_{\mathbb R}\frac{\partial}{\partial s}W_s(x,y)_{|s=t^2}dy\not =0,\quad x\in\mathbb{R}\mbox{ and }t>0,
$$
(see \cite[Proposition 5.1]{StTo4}).

The Laguerre differential operator $L_\alpha$, $\alpha >-1/2$, can
be written by
$$L_{\alpha}=\frac{1}{2}\Bigl(-\frac{d^{2}}{dx^{2}}+x^{2}+\frac{1}{x^{2}}
\Big(\alpha^{2}-\frac{1}{4}\Big)\Bigr).
$$

\noindent This operator is positive and symmetric in the domain
$C_{c}^{\infty}(0,\infty)$ with respect to $L^{2}(0,\infty)$.
Here $C_{c}^{\infty}(0,\infty)$ denotes the space of
$C^{\infty}$-functions that have compact support on $(0,\infty)$.
For every $n\in\mathbb{N}$, one has
$$
L_{\alpha}\varphi_{n}^{\alpha}=(2n+\alpha+1)\varphi_{n}^{\alpha},
$$
where
$$
\varphi_{n}^{\alpha}(x)=\Biggl(\frac{2\Gamma(n+1)}{\Gamma(n+1+\alpha)}\Biggr)^{1/2}e^{-x^2/2}
\:x^{\alpha+\frac{1}{2}}\:L_{n}^{\alpha}(x^{2}), \quad\:x\in
(0,\infty),$$ and $L_{n}^{\alpha}$ is the Laguerre polinomial of
type $\alpha$  (\cite[p. 100]{Sz} and \cite[p. 7]{Th2}). The heat
diffusion semigroup $\mathbb{W}^\alpha =\{W^{\alpha}_{t}\}_{t>0}$ generated
by $L_{\alpha}$ is defined by
$$
W^{\alpha}_{t}(f)=\sum_{n=0}^{\infty}\:c^{\alpha}_{n}(f)\:e^{-(2n+\alpha+1)t}\:\varphi_{n}^{\alpha},
\quad f\in L^{2}(0,\infty),
$$
being
$$
c_{n}^{\alpha}(f)=\int_{0}^{\infty}\varphi_{n}^{\alpha}(x)f(x)\:dx,\quad
n\in\mathbb{N}.
$$

According to the Mehler formula for Laguerre polynomials
(\cite[(1.1.47)]{Th2}) we can write, for $f\in L^2(0,\infty )$,
$$
W^{\alpha}_{t}(f)(x)=\int_{0}^{\infty}W^{\alpha}_{t}(x,y)f(y)\:dy,\quad x\in (0,\infty),
$$
where, for each $x,y,t\in (0,\infty)$,
$$
W^{\alpha}_{t}(x,y)=\Biggl(\frac{2e^{-t}}{1-e^{-2t}}\Biggr)^{1/2}\Biggl(\frac{2xye^{-t}}{1-e^{-2t}}\Biggr)^{1/2}
I_{\alpha}\Biggl(\frac{2xye^{-t}}{1-e^{-2t}}\Biggr)e^{-\frac{1}{2}(x^{2}+y^{2})\frac{1+e^{-2t}}{1-e^{-2t}}}.
$$
Here $I_{\alpha}$ denotes the modified Bessel function of the
first kind and order $\alpha$.

Muckenhoupt (\cite{Mu2}) investigated harmonic analysis associated
with Laguerre polynomials $\{L_{n}^{\alpha}\}_{n\in \mathbb{N}}$.
More recently, harmonic analysis operators in the
$L_{\alpha}$-setting have been studied by several authors. Stempak
\cite{St8}, Mac\'ias, Segovia and Torrea \cite{MST2} and Chicco
Ruiz and Harboure \cite{CH} studied maximal operator for the heat
semigroup $\mathbb{W}^\alpha =\{W^{\alpha}_{t}\}_{t>0}$. Riesz transforms
associated with Laguerre functions was investigated by Harboure,
Torrea and Viviani (\cite{HTV2}) and Harboure, Segovia, Torrea and
Viviani (\cite{HSTV}). Also, the papers of Nowak (\cite{No1} and
\cite{No2}) and Nowak and Stempak (\cite{NS1} and \cite{NS2}) are
remarkable. In \cite{GIT}, Guti\'errez, Incognito and Torrea
investigated Riesz transforms and certain Littlewood-Paley
functions in the Laguerre context by exploiting a relation between
n-dimensional Hermite polynomials and Laguerre polynomials when
$\alpha=\frac{n-1}{2}$. This idea was also used in \cite{GLLNU} to
study higher order Riesz transform associated with Laguerre
polynomials. In \cite{BFRST1}, Betancor, Fari\~{n}a, Rodr\'{\i}guez-Mesa,
Sanabria and Torrea developed a new procedure to analyze operators
in the Laguerre setting. The operator under consideration is
decomposed into a local part and into a global part. The
boundedness pro\-per\-ties of the local operator are deduced from the
boundedness properties of the corresponding operator in the
Hermite setting. This transference procedure works for every value
of $\alpha$ and it uses properties of operators in the one
dimension Hermite context in contrast with the method employed in
\cite{GLLNU} and \cite{GIT}. In this paper we apply the procedure
introduced in \cite{BFRST1} to es\-ta\-blish $L^{p}$-boundedness
properties for the area g-functions $g_{\mathbb{W}^\alpha}^q$, $q>1$, associated
with the heat Laguerre semigroup
$\mathbb{W}^\alpha=\{W_{t}^{\alpha}\}_{t>0}$, and defined by
$$
g_{\mathbb{W}^\alpha}^q(f)(x)=\Biggl\{\int_{\Gamma_{+}(x)}\biggl|\left(s\frac{\partial}{\partial
s} W_{s}^{\alpha}(f)(y)\right)_{|s=t^{2}}\biggr|^q \frac{dt\:dy}{t^{2}}
\Biggr\}^{1/q},\quad x\in (0,\infty ),
$$
where $\Gamma_{+}(x)=\{(y,t)\in
(0,\infty)\times(0,\infty):|y-x|<t\}$, for every $x\in
(0,\infty)$. We transfer the result obtained in Proposition
\ref{Hermite} from $g_{\mathbb{W}}^q$ to $g_{\mathbb{W}^\alpha}^q$ and we get the
following.
\begin{propo}\label{Laguerre}
Let $\alpha >-1/2$ and $q\geq 2$. Then $g_{\mathbb{W}^\alpha}^q$ defines a
bounded operator from $L^{p}(0,\infty)$ into itself, for
every $1<p<\infty$, and from $L^{1}(0,\infty)$ into
$L^{1,\infty}(0,\infty)$.
\end{propo}

We complete this result with the next reverse $L^p$-boundedness property for $g_{\mathbb{W}^\alpha}^q$, $1<q\leq2$.
\begin{propo}\label{reverse}
Let $\alpha >-1/2$ and $1<q\leq2$. For every $1<p<\infty$, there exists $C>0$ for which
$$
||f||_{L^p(0,\infty )}\leq C||g_{\mathbb{W}^\alpha }^q(f)||_{L^p(0,\infty )},\quad f\in L^p(0,\infty ).
$$
\end{propo}
\noindent By combining Propositions \ref{Laguerre} and \ref{reverse} we can obtain the following.
\begin{coro}
Let $\alpha >-1/2$ and $1<p<\infty $. Then, there exists $C>0$ such that
$$
\frac{1}{C}||f||_{L^p(0,\infty )}\leq ||g_{\mathbb{W}^\alpha }^2(f)||_{L^p(0,\infty )}\leq C||f||_{L^p(0,\infty )},\quad f\in L^p(0,\infty ).
$$
\end{coro}

Hardy spaces associated with Schr\"odinger operators were studied by Dziuba\'nski and Zien\-kie\-wicz (\cite{DZ1}, \cite{DZ2} and \cite{DZ3}). The Hermite operator is a special case of the operators considered by these authors. The Hardy space $H^1_{\mathcal H}(\mathbb{R})$ in the Hermite setting consists of all those functions $f\in L^1(\mathbb{R})$ such that $\sup _{t>0}|W_t(f)|\in L^1(\mathbb{R})$. The norm $||\cdot ||_{H^1_{\mathcal H}(\mathbb{R})}$ in $H^1_{\mathcal{H}}(\mathbb{R})$ is defined by
$$
||f||_{H^1_{\mathcal H}(\mathbb{R})}=\left\|\sup_{t>0}|W_t(f)|\right\|_{L^1(\mathbb{R})},\quad f\in H^1_{\mathcal H}(\mathbb{R}).
$$

Hardy spaces in the Laguerre context have been investigated by Dziuba\'nski (\cite{Dz1} and \cite{Dz2}). A function $f\in L^1(0,\infty )$ is in the Hardy space $H^1_{L_\alpha }(0,\infty )$ when $\sup_{t>0}|W_t^\alpha (f)|\in L^1(0,\infty )$. The  norm $\|\cdot\|_{H^1_{L_\alpha }(0,\infty )}$ is given by
$$
||f||_{H^1_{L_\alpha }(0,\infty )}=\left\|\sup_{t>0}|W_t^\alpha (f)|\right\|_{L^1(0,\infty )},\quad f\in H^1_{L_\alpha }(0,\infty ).
$$

Recently, Betancor, Dziuba\'nski and Garrig\'os (\cite{BDG}) have established a useful connection between the spaces $H^1_{\mathcal H}(\mathbb{R})$ and $H^1_{L_\alpha }(0,\infty )$. Let $f\in L^1(0,\infty )$. We denote by $f_{\rm o}$ the odd extension of $f$ to $\mathbb{R}$. Then, $f\in H^1_{L_\alpha}(0,\infty )$ if, and only if, $f_{\rm o}\in H^1_{\mathcal H}(\mathbb{R})$. Moreover, the quantities $||f||_{H^1_{L_\alpha }(0,\infty )}$ and $||f||_{H^1_{\mathcal H}(\mathbb{R})}$ are equivalent.

It is known that a function $f\in L^1(\mathbb{R})$ is in $H^1_{\mathcal H}(\mathbb{R})$ when and only when $g^2_{\mathbb{W}}(f)\in L^1(\mathbb{R})$ (\cite[Proposition 4, p. 124]{Stein}). In the following we establish the corresponding result in the Laguerre setting.

\begin{propo}\label{H1}
Let $\alpha >-1/2$. Then, if $f\in L^1(0,\infty )$, $f\in H^1_{L_\alpha }(0,\infty )$ if and only if, $g^2_{\mathbb{W}^\alpha }(f)\in L^1(0,\infty )$. Moreover, there exists $C>0$ such that
$$
\frac{1}{C}||f||_{H^1_{L_\alpha }(0,\infty )}\leq ||f||_{L^1(0,\infty )}+||g_{{\mathbb W}^\alpha }^2(f)||_{L^1(0,\infty )}\leq C||f||_{H^1_{L_\alpha }(0,\infty )},\quad f\in H^1_{L_\alpha }(0,\infty ).
$$
\end{propo}

In \cite{Xu} Xu studied the Littlewood-Paley theory for functions with values in Banach spaces. He characterized the uniform convexity and smoothness of the underlying Banach spaces by the validity in the vector valued setting of the $L^p$-inequalities for the Lusin area integral associated with the Poisson semigroup on the unit disc on the complex plane. The procedure developed in this paper allows us to obtain the results in Propositions \ref{Hermite}, \ref{Laguerre} and \ref{reverse} when the heat semigroup is replaced by the Poisson semigroup in the Hermite and Laguerre setting. We can also obtain new characterizations of the uniform convexity and smoothness of a Banach space in terms of the $L^p$-inequalities of the area Littlewood-Paley functions for the Poisson semigroup in the Laguerre setting by using \cite[Theorem 2.1]{MTX} (see also \cite{BFRST2}).

\noindent The next useful properties for the Bessel functions $I_{\alpha}$, $\alpha >-1/2$, can be found in \cite[Ch. 5]{Leb}.

\noindent ({\bf I1}) $I_{\alpha}(z)\sim z^{\alpha},\:\:{\mbox as}\:\: z\rightarrow
0^{+}$.

\noindent ({\bf I2})
$$
\sqrt{z}I_{\alpha}(z)=\frac{1}{\sqrt{2\pi}}e^{z}\Bigl(
\sum_{k=0}^{n}(-1)^{k}[\alpha,k](2z)^{-k}+O(z^{-n-1})\Bigr),\:\:{\mbox
as}\:\: z\rightarrow \infty,
$$
where $[\alpha,0]=1$, and
$$
[\alpha,k]=\frac{(4\alpha^{2}-1)(4\alpha^{2}-3^{2})\dots (4\alpha^{2}-(2k-1)^{2})}{2^{2k}\Gamma(k+1)}, \:\: k=1,2,\dots,
$$

\noindent ({\bf I3}) $\frac{d}{dz}(z^{-\alpha}I_{\alpha}(z))=z^{-\alpha}I_{\alpha+1}(z),
\:\: z\in (0,\infty).$

\noindent We also state here the following elementary properties which will be often used along the paper.

\noindent ({\bf P1}) For every $x,y\in\mathbb{R}$ and $s>0$, one has

({\bf a}) $(x-e^{-s}y)^2+(y-e^{-s}x)^2=2(x-y)^2e^{-s}+(x^2+y^2)(1-e^{-s})^2$.

({\bf b}) $(x-e^{-s}y)^2+(y-e^{-s}x)^2=(x-y)^2(1+e^{-2s})+2xy(1-e^{-s})^2=(x^2+y^2)(1+e^{-2s})-4xye^{-s}$.

({\bf c}) $(x-e^{-s}y)^2+(y-e^{-s}x)^2\geq \frac{(x-y)^2}{2}$.

\noindent ({\bf P2}) $\frac{2}{u}\leq \frac{1+e^{-u}}{1-e^{-u}}\leq \frac{2}{1-e^{-u}}$, $u>0$.

\noindent ({\bf P3}) For each $a>0$ there exists $c>0$ such that $ue^{-au}\leq c(1-e^{-u})$, $u>0$.

\noindent ({\bf P4}) For every $a>0$ and $b\geq 0$, it can be found $c>0$ for which $u^be^{-au}\leq c$, $u>0$.

Throughout this paper by $C$ we always denote a positive constant
that can change from a line to another one.

\section{Proof of Proposition \ref{Hermite}}

Let us fix $q\geq 2$. We will use the vector-valued Calder\'on-Zygmund theory
to see the $L^p$-boundedness properties of the g-function $g_\mathbb{W}^q$.

We first observe that we can write, for every $f\in L^q(\mathbb{R}
)$,
$$
g_\mathbb{W}^q(f)(x)=\left\|\left(s\frac{\partial }{\partial
s}W_s(f)(x+y)\right)_{|s=t^2}\right\|_{L^q(\Gamma
(0),\frac{dtdy}{t^2})},\quad x\in \mathbb{R} .
$$

Moreover, if $f\in L^q(\mathbb{R} )$ then
\begin{equation}\label{interchange}
\frac{\partial }{\partial s}W_s(f)(x)=\int_{-\infty }^{+\infty
}\frac{\partial }{\partial s}W_s(x,y)f(y)dy,\quad x\in \mathbb{R}
\mbox{ and }s>0.
\end{equation}

Indeed, let $f\in L^q(\mathbb{R} )$. By using (\ref{kernelW}) we
can write, for every $x,y \in \mathbb{R}$ and $s>0$,
\begin{eqnarray}\label{partials}
\frac{\partial }{\partial s}W_s(x,y)&=&-\frac{1}{2\sqrt{\pi
}}e^{-\frac{(x-e^{-s}y)^2+(y-e^{-s}x)^2}{2(1-e^{-2s})}}\frac{e^{-s/2}}{(1-e^{-2s})^{3/2}}\nonumber\\
&\times&\left\{1+e^{-2s}+2e^{-s}(y(x-e^{-s}y)+x(y-e^{-s}x))-2e^{-2s}\frac{(x-e^{-s}y)^2+(y-e^{-s}x)^2}
{1-e^{-2s}}\right\}.
\end{eqnarray}
By taking into account (P1)(a) and (P4), a straightforward manipulation leads to
\begin{eqnarray}\label{acotpartials}
\left|\frac{\partial }{\partial s}W_s(x,y)\right|&\leq
&Ce^{-\frac{(x-e^{-s}y)^2+(y-e^{-s}x)^2}
{4(1-e^{-2s})}}\frac{e^{-s/2}}{(1-e^{-2s})^{3/2}}\left(1+(|x|+|y|)
e^{-s}(1-e^{-2s})^{1/2}\right)\nonumber\\
&\leq& Ce^{-\frac{(x-e^{-s}y)^2+(y-e^{-s}x)^2}
{8(1-e^{-2s})}}\frac{e^{-s/2}}{(1-e^{-2s})^{3/2}},\quad x,y\in
\mathbb{R} \mbox{ and }s>0.
\end{eqnarray}
Moreover, if $s_0>0$ is fixed, by (P1)(c) and (P2) we obtain, for each $x,y\in \mathbb{R}$,
$$
\left|\frac{\partial }{\partial s}W_s(x,y)\right|\leq
Ce^{-\frac{(x-y)^2}
{32s}}\frac{e^{-s/2}}{(1-e^{-2s})^{3/2}}\leq Ce^{-\frac{(x-y)^2}
{64s_0}}\frac{e^{-s_0/4}}{(1-e^{-s_0})^{3/2}}, \quad \frac{s_0}{2}<s<2s_0.
$$
Since $f\in L^q(\mathbb{R})$, by the H\"older inequality one gets
$$
\int_{-\infty }^{+\infty }e^{-\frac{(x-y)^2}
{64s_0}}\frac{e^{-s_0/4}}{(1-e^{-s_0})^{3/2}}|f(y)|dy<\infty ,\quad x\in \mathbb{R}.
$$
Then, the mean value theorem and dominated convergence theorem lead to (\ref{interchange}).

Let us define the operator $T$ on $L^q(\mathbb{R})$ as follows
$$
[Tf(x)](y,t)=\int_{-\infty }^{+\infty }K(x,z)(y,t)f(z)dz, \quad x,y\in \mathbb{R} \mbox{ and
}t>0,
$$
where $K(x,z)$, $x,z\in \mathbb{R}$, denotes the function defined
on $\Gamma (0)$ by
$$
K(x,z)(y,t)=\left(s\frac{\partial }{\partial s}W_s(x+y,z)\right)_{|s=t^2}.
$$

Note that for every $f\in L^q(\mathbb{R})$, $g_\mathbb{W}^q(f)(x)=||Tf(x)||_{L^q(\Gamma (0),\frac{dydt}{t^2})}$, $x\in \mathbb{R}$.

Our objective then is to prove the following assertions:

(i) $T$ is a bounded operator from $L^q(\mathbb{R} )$ into
$L^q_{L^q(\Gamma (0),\frac{dydt}{t^2})}(\mathbb{R} )$, or equivalently, $g_\mathbb{W}^q$ is bounded from $L^q(\mathbb{R})$ into itself.

(ii) For every $x,z\in \mathbb{R}$, $x\not=z$,

\hspace{1.5em} $\displaystyle ||K(x,z)||_{L^q(\Gamma
(0),\frac{dtdy}{t^2})}\leq \frac{C}{|x-z|}$, and

\hspace{1.5em} $\displaystyle \Big\|\frac{\partial}{\partial
x}K(x,z)\Big\|_{L^q(\Gamma (0),\frac{dtdy}{t^2})}+\Big\|\frac
{\partial }{\partial z}K(x,z)\Big\|_{L^q(\Gamma
(0),\frac{dtdy}{t^2})}\leq \frac{C}{|x-z|^2}$.\vspace{0.2cm}

(iii) For every $f\in L^q(\mathbb{R})$,
$$
Tf(x)=\int_{-\infty }^{+\infty }K(x,y)f(y)dy, \quad x\not \in \mbox{ supp }f.
$$
Then the vector valued Calder\'on-Zygmund theory allows us to
obtain the desired result.

Let us establish (i). Consider, for each $r>1$, the operator $g_r$ given by
$$
g_r(f)(x)=\left\{\int_0^\infty \left|t\frac{\partial}{\partial t}W_t(f)(x)\right|^r\frac{dt}{t}\right\}^{1/r},\quad x\in \mathbb{R}.
$$

$L^p$-boundedness properties for $g_2$ were analyzed in \cite{Th2} and also in \cite{StTo2}. In \cite[Theorem 2.2]{StTo2} it was established that $f\longrightarrow g_2(f)$ is bounded from $L^p(\mathbb{R})$ into itself, for every $1<p<\infty$. Thus, in particular $g_2$ is a bounded operator from $L^q(\mathbb{R})$ into itself.

We can see that $g_q$ is also a bounded operator from $L^q(\mathbb{R})$ into itself.
In effect, since $q\geq 2$ one has
$$
||g_q(f)||_{L^q(\mathbb{R})}^q\leq\int_{-\infty }^{+\infty }\int_0^\infty \left|t\frac{\partial}{\partial t}W_t(f)(x)\right|^2\frac{dt}{t}\left(\sup_{t>0}\left|t\frac{\partial}{\partial t}W_t(f)(x)\right|\right)^{q-2}dx.
$$

Also, from (\ref{acotpartials}) and using (P1)(c) and (P3) we obtain
$$
\left|t\frac{\partial}{\partial t}W_t(f)(x)\right|\leq C\int_{-\infty }^{+\infty }|f(y)|\frac{t^{3/2}e^{-t/2}}{(1-e^{-2t})^{3/2}}\frac{e^{-\frac{(x-y)^2}{32t}}}{\sqrt{t}}dy\leq C\mathcal{W}_*(|f|)(x),\quad x\in \mathbb{R},
$$
where $\mathcal{W}_*$ represents the maximal operator associated to the heat semigroup for the Laplacian
operator, that is,
$$
\mathcal{W}_*(f)(x)=\sup_{t>0}\left|\frac{1}{2\sqrt{\pi}}\int_{-\infty }^{+\infty }f(y)\frac{e^{-\frac{(x-y)^2}{4t}}}{\sqrt{t}}dy\right|.
$$
Then, by the H\"older inequality and the fact that $g_2$ and $\mathcal{W}_*$ are bounded operators from $L^q(\mathbb{R})$ into itself, we get
\begin{eqnarray*}
||g_q(f)||_{L^q(\mathbb{R})}^q&\leq &C\int_{-\infty }^{+\infty }(g_2(f)(x))^2(\mathcal{W}_*(|f|)(x))^{q-2}dx\\
&\leq&C\left(\int_{-\infty}^{+\infty }|g_2(f)(x)|^qdx\right)^{2/q}\left(\int_{-\infty }^{+\infty}(\mathcal{W}_*(|f|)(x))^qdx\right)^{(q-2)/q}\\
&=&C||g_2(f)||_{L^q(\mathbb{R})}^2||\mathcal{W}_*(|f|)||_{L^q(\mathbb{R})}^{q-2}\leq C||f||_{L^q(\mathbb{R})}^q.
\end{eqnarray*}

Finally, we observe that
\begin{eqnarray*}
||g_\mathbb{W}^q(f)||_{L^q(\mathbb{R})}^q&=&\int_0^\infty \int_{-\infty }^{+\infty }\left|\left(s\frac{\partial}{\partial s}W_s(f)(y)\right)_{|s=t^2}\right|^q\int_{y-t}^{y+t}dxdy\frac{dt}{t^2}\\
&=&2\int_0^\infty \int_{-\infty }^{+\infty}\left|\left(s\frac{\partial}{\partial s}W_s(f)(y)\right)_{|s=t^2}\right|^qdy\frac{dt}{t}.
\end{eqnarray*}

By making the change of variables $u=t^2$, one gets
$$
||g_\mathbb{W}^q(f)||_{L^q(\mathbb{R})}^q=\int_0^\infty \int_{-\infty }^{+\infty }\left|u\frac{\partial}{\partial u}W_u(f)(y)\right|^qdy\frac{du}{u}=||g_q(f)||_{L^q(\mathbb{R})}^q.
$$
Thus, (i) is proved.

We now establish (ii).  First, let us prove that
\begin{equation}\label{CZ1}
||K(x,z)||_{L^q(\Gamma (0),\frac{dtdy}{t^2})}=\left\{\int_{\Gamma
(x)}\Big|\left(s\frac{\partial }{\partial
s}W_s(y,z)\right)_{|s=t^2}\Big|^q\frac{dtdy}{t^2}\right\}^{1/q}\leq
\frac{C}{|x-z|},\quad x,z\in \mathbb{R}, \;x\not=z.
\end{equation}

From (\ref{acotpartials}) and taking into account (P1)(c), (P3) and (P4) we can write
\begin{eqnarray*}
\int_{\Gamma (x)}\Big|s\frac{\partial }{\partial
s}W_s(y,z)_{|s=t^2}\Big|^q\frac{dtdy}{t^2}&\leq&C\int_{-\infty}^{+\infty
}\int_{|x-y|}^{+\infty }t^{2q-2}
e^{-\frac{q(z-y)^2}{32t^2}}\frac{e^{-\frac{qt^2}{2}}}{(1-e^{-2t^2})^{\frac{3q}{2}}}dtdy\\
&\hspace{-8cm}\leq &\hspace{-4cm}C\int_{-\infty}^{+\infty }\int_{|x-y|}^{+\infty }
\left(\frac{t^2e^{-t^2/3}}{1-e^{-2t^2}}\right)^{\frac{3q}{2}} \frac{1}{(t^2+(z-y)^2)^{\frac{q}{2}+1}}dtdy\\
&\hspace{-8cm}\leq &\hspace{-4cm}C\int_{-\infty}^{+\infty
}\int_{|x-y|}^{+\infty }\frac{1}{(t+|z-y|)^{q+2}}dy
\leq C\int_{-\infty}^{+\infty }\frac{1}{(|x-y|+|z-y|)^{q+1}}dy\\
&\hspace{-8cm}\leq&\hspace{-4cm}
C\left(\int_{I_{x,z}}\frac{1}{|z-x|^{q+1}}dy+\int_{\mathbb{R}
\setminus I_{x,z}}\frac{1}{|2y-x-z|^{q+1}}dy\right)\leq
\frac{C}{|x-z|^q},\quad x,z\in \mathbb{R},\;x\not=z.
\end{eqnarray*}
Here $I_{x,z}$ represents the interval $I_{x,z}=(\min\{x,z\}, \max\{x,z\})$. Thus (\ref{CZ1}) is established.

We now see that
\begin{equation}\label{CZ2}
\left|\left|\frac{\partial }{\partial
x}K(x,z)\right|\right|_{L^q(\Gamma (0),\frac{dtdy}{t^2})}+
\left|\left|\frac{\partial }{\partial
z}K(x,z)\right|\right|_{L^q(\Gamma (0),\frac{dtdy}{t^2})}\leq
\frac{C}{|x-z|^2},\quad x,z\in \mathbb{R},\;x\not=z.
\end{equation}

We will show that, when $x,z\in \mathbb{R}$, $x\not=z$,
\begin{equation}\label{CZ2a}
\left|\left|\frac{\partial }{\partial
x}K(x,z)\right|\right|_{L^q(\Gamma (0),\frac{dtdy}{t^2})}=\left\{\int_{\Gamma
(x)}\Big|\left(s\frac{\partial ^2}{\partial y\partial
s}W_s(y,z)\right)_{|s=t^2}\Big|^q\frac{dtdy}{t^2}\right\}^{1/q} \leq
\frac{C}{|x-z|^2}.
\end{equation}
The analogous property for $\left\|\frac{\partial }{\partial
z}K(x,z)\right\|_{L^q(\Gamma (0),\frac{dtdy}{t^2})}$ can be established in a similar way.

From (\ref{partials}) we have that, for every $y,z\in
\mathbb{R}$ and $s>0$,
\begin{eqnarray*}
\frac{\partial ^2}{\partial y\partial
s}W_s(y,z)&=&-\frac{(y-e^{-s}z)-(z-e^{-s}y)e^{-s}}{1-e^{-2s}}\frac{\partial}{\partial s}W_s(y,z)\\
&\hspace{-4cm}-&\hspace{-2cm}\frac{2}{\sqrt{\pi
}}e^{-\frac{(y-e^{-s}z)^2+(z-e^{-s}y)^2}{2(1-e^{-2s})}}\frac{e^{-s/2}}{(1-e^{-2s})^{3/2}}
\left(e^{-s}(z-e^{-s}y)-\frac{e^{-2s}[y-e^{-s}z-e^{-s}(z-e^{-s}y)]}{1-e^{-2s}}\right).
\end{eqnarray*}

Hence, by using (\ref{acotpartials}) we obtain
\begin{equation}\label{D2WHermite}
\left|\frac{\partial ^2}{\partial y\partial s}W_s(y,z)\right|\leq
Ce^{-\frac{(y-e^{-s}z)^2+(z-e^{-s}y)^2}{16(1-e^{-2s})}}\frac{e^{-s/2}}{(1-e^{-2s})^2},\quad
y,z\in \mathbb{R} ,s>0,
\end{equation}
and by proceeding as in the proof of (\ref{CZ1}) we conclude
(\ref{CZ2a}).

To finish we prove (iii). Let $f\in L^q(\mathbb{R})$ and let $\mathcal{K}=\mbox{ supp }f$. According to (ii) and the Minkowski inequality we can write
\begin{eqnarray*}
\left\|\int_{-\infty }^{+\infty }K(x,z)f(z)dz\right\|_{L^q(\Gamma (0),\frac{dtdy}{t^2})}&\leq &\int_{-\infty }^\infty
\left\|K(x,z)\right\|_{L^q(\Gamma (0),\frac{dtdy}{t^2})}|f(z)|dz\\
&\hspace{-10cm}\leq&\hspace{-5cm}C\int_{-\infty }^{+\infty }\frac{1}{|x-z|}|f(z)|dz\leq C\left(\int_\mathcal{K}\frac{1}{|x-z|^{q'}}dz\right)^{1/q'}\|f\|_{L^q(\mathbb{R})}<\infty ,
\end{eqnarray*}
for every $x\not\in \mathcal{K}$. Assume that $g\in L^{q'}\left(\Gamma (0),\frac{dtdy}{t^2}\right)$. We have
\begin{eqnarray*}
\int_{\Gamma (0)}g(y,t)\left(\int_{-\infty }^{+\infty }K(x,z)f(z)dz\right)(y,t)\frac{dtdy}{t^2}&=&\int_{-\infty }^{+\infty }f(z)\int_{\Gamma (0)}K(x,z)(y,t)g(y,t)\frac{dtdy}{t^2}dz\\
&\hspace{-10cm}=&\hspace{-5cm}\int_{\Gamma (0)}g(y,t)\int_{-\infty }^{+\infty }K(x,z)(y,t)f(z)dz\frac{dtdy}{t^2},\quad x\not\in \mathcal{K}.
\end{eqnarray*}
These equalities are justified by taking into account that
\begin{eqnarray*}
\int_{\Gamma (0)}|g(y,t)|\left(\int_{-\infty }^{+\infty }|K(x,z)(y,t)||f(z)|dz\right)\frac{dtdy}{t^2}&&\\
&\hspace{-10cm}\leq &\hspace{-5cm}||g||_{L^{q'}(\Gamma (0),\frac{dtdy}{t^2})}||f||_{L^q(\mathbb{R})}\left(\int_{\mathcal{K}}\frac{1}{|x-z|^{q'}}dz\right)^{1/q'}<\infty ,\quad x\not \in \mathcal{K}.
\end{eqnarray*}
Hence
$$
Tf(x)=\int_{-\infty }^{+\infty }K(x,z)f(z)dz, \quad x\not \in \mathcal{K}.
$$
\section{Proof of Proposition \ref{Laguerre}}
In this section we
exploit the arguments developed in \cite{BFRST1} where the basic idea
is to compare in some region the g-functions $g_{{\mathbb W}^{\alpha}}^q$ and
$g_{\mathbb W}^q$ and then to tranfer $L^{p}$-boundedness properties from
$g_{\mathbb{W}}^q$ to $g_{\mathbb{W}^{\alpha}}^q$. We
first observe that if $f_{\rm o}$ denotes the odd extension to $\mathbb{R}$ of a suitable function $f$ defined on $(0,\infty)$, we can write
\begin{equation}\label{Ws}
W_{s}(f_{\rm o})(x)=\int_{0}^{\infty}(W_{s}(x,y)-W_{s}(x,-y))f(y)dy,\quad
x\in \mathbb{R}\mbox{ and }s>0.
\end{equation}
Then $W_s(f_{\rm o})$ is odd and $g_\mathbb{W}^q(f_{\rm o})$ is even. Thus we
get
\begin{eqnarray*}
g_\mathbb{W}^q(f_{\rm o})(x)=g_\mathbb{W}^q(f_{\rm o})(|x|)&=&\left\{\int_{\Gamma
(|x|)}\Big|s\frac{\partial }{\partial
s}W_s(f_{\rm o})(y)_{|s=t^2}\Big|^q\frac{dtdy}{t^2}\right\}^{1/q}\\
&\leq& \left\{2\int_{\Gamma _+(|x|)}\Big|s\frac{\partial
}{\partial
s}W_s(f_{\rm o})(y)_{|s=t^2}\Big|^q\frac{dtdy}{t^2}\right\}^{1/q},\quad
x\in \mathbb{R} .
\end{eqnarray*}
Moreover, it is clear that
$$
\left\{\int_{\Gamma _+(|x|)}\Big|s\frac{\partial }{\partial
s}W_s(f_{\rm o})(y)_{|s=t^2}\Big|^q\frac{dtdy}{t^2}\right\}^{1/q}\leq
g_\mathbb{W}^q(f_{\rm o})(x), \quad  x\in \mathbb{R}.
$$
Hence, according to Proposition \ref{Hermite}, the operator
$g_\mathbb{W}^{q,+}$ defined by
\begin{equation}\label{gq+}
g_\mathbb{W}^{q,+}(f)(x)=\left\{\int_{\Gamma _+(x)}\Big|s\frac{\partial
}{\partial
s}W_s(f_{\rm o})(y)_{|s=t^2}\Big|^q\frac{dtdy}{t^2}\right\}^{1/q    },\quad
x\in (0,\infty ),
\end{equation}
is bounded from $L^p(0,\infty )$ into itself, for every
$1<p<\infty$, and from $L^1(0,\infty )$ into $L^{1,\infty
}(0,\infty )$.

Then, the proof of Proposition \ref{Laguerre} will be finished
when we establish the $L^p$-boundedness properties for the
operator $D_{q,\alpha }=g_{\mathbb{W}^\alpha}^q-g_\mathbb{W}^{q,+} $. Let us describe the main steps
we follow.

By taking into account (\ref{Ws}) and using triangular inequality in $L^q(\Gamma (0),\frac{dtdy}{t^2})$ we write
\begin{eqnarray*}
|D_{q,\alpha }(f)(x)|&\leq &\left\{\int_{\Gamma _+(x)}\left|\int_0^\infty \left[s\frac{\partial }{\partial s}(W_s^\alpha (y,z)
-W_s(y,z))\right]_{|s=t^2}f(z)dz\right|^q\frac{dtdy}{t^2}\right\}^{1/q}\\
&+&\left\{\int_{\Gamma _+(x)}\left|\int_0^\infty \left[s\frac{\partial }{\partial s}(W_s(y,-z))\right]_{|s=t^2}
f(z)dz\right|^q\frac{dtdy}{t^2}\right\}^{1/q},\quad x\in (0,\infty ).
\end{eqnarray*}

We split the first integral in three parts and, by using Minkowsky inequality, we get
\begin{eqnarray*}
 |D_{q,\alpha }(f)(x)|&\leq &\Biggl\{\int_{\Gamma_{+}(x)}\biggl|\biggl(\int_{0}^{\frac{x}{2}}+\int_{\frac{x}{2}}^{2x}
 +\int_{2x}^{\infty}\biggr)\left[s\frac{\partial }{\partial s}(W_s^\alpha (y,z)-W_s(y,z))\right]_{|s=t^2}f(z)dz\biggr|^q \frac{dtdy}{t^{2}}
                             \Biggr\}^{1/q}\\
&+&\Biggl\{\int_{\Gamma_{+}(x)}\biggl|\int_{0}^{\infty}\left[s\frac{\partial }{\partial s}
(W_s(y,-z))\right]_{|s=t^2}f(z)dz\biggr|^q\frac{dtdy}{t^{2}}
                             \Biggr\}^{1/q}\\
&\leq&\left(\int_{(0,\infty )\setminus  (\frac{x}{2},2x)}\right)\left[
\Biggl\{\int_{\Gamma_{+}(x)}\biggl|\left(s\frac{\partial}{\partial s}
                             W_{s}^{\alpha}(y,z)\right)_{|s=t^{2}}\biggr|^q\:\frac{dtdy}{t^{2}}
                             \Biggr\}^{1/q}\right.\\
&&\hspace{3cm}+\left.\Biggl\{\int_{\Gamma_{+}(x)}\biggl|\left(s\frac{\partial}{\partial s}
                             W_{s}(y,z)\right)_{|s=t^{2}}\biggr|^q\:\frac{dtdy}{t^{2}}
                             \Biggr\}^{1/q}\right]|f(z)|dz\\
&+&\int_0^\infty \left\{\int_{\Gamma _+(x)}\left|\left(s\frac{\partial }{\partial s}
(W_s(y,-z))\right)_{|s=t^2}\right|^q\frac{dtdy}{t^2}\right\}^{1/q}|f(z)|dz\\
&+&\int_{\frac{x}{2}}^{2x}\Biggl\{\int_{\Gamma_{+}(x)}\biggl|\left(s\frac{\partial
}{\partial s}(W_s^\alpha (y,z)-W_s(y,z))\right)_{|s=t^2}\biggr|^q
\frac{dtdy}{t^{2}}
                             \Biggr\}^{1/q}|f(z)|dz\\[0.2cm]
&=&A_{1}(|f|)(x)+A_{2}(|f|)(x)+A_{3}(|f|)(x),\quad x\in (0,\infty).
\end{eqnarray*}

Suitable estimates for the kernels of the operators $A_{1}$ and
$A_{2}$ allow us to mayore $A_{1}$ and $A_{2}$  by certain Hardy
type operators. Then the boundedness $L^{p}$-properties of $A_{1}$
and $A_{2}$ are obtained from the corresponding ones for Hardy
type operators (see for instance, \cite[p. 20]{Zyg}). To prove
that $A_{3}$ is bounded from $L^{p}(0,\infty)$ into itself, for
each $1\leq p\leq\infty$, we will find a nonnegative function $H(x,z)$,
$0<\frac{x}{2}<z<2x<\infty$, verifying that
$$
\Biggl\{\int_{\Gamma_{+}(x)}\biggl|\left(s\frac{\partial}{\partial s}
                             [W_{s}^{\alpha}(y,z)-W_{s}(y,z)]\right)_{|s=t^{2}}\biggr|^q \frac{dtdy}{t^{2}}
                             \Biggr\}^{1/q}\leq H(x,z),
$$
and such that the operator $\mathfrak{B}$ defined by
$$
\mathfrak{B}(f)(x)=\int_{\frac{x}{2}}^{2x}H(x,z)f(z)dz
$$
is bounded from $L^{p}(0,\infty)$ into itself, for every $1\leq p\leq \infty$.

Next we study the operators that we have defined above.

First, we obtain some estimations which will be used later.
According to (I3) we can write, for every $s,y,z\in
(0,\infty )$,
\begin{eqnarray}\label{dWs}
\frac{\partial }{\partial s}W_s^\alpha (y,z)&=&\frac{\partial
}{\partial
s}\left[2^{\alpha +1}(yz)^{\alpha +1/2}\left(\frac{e^{-s}}{1-e^{-2s}}\right)^{\alpha
+1}e^{-\frac{1}{2}(y^2+z^2)\frac{1+e^{-2s}}{1-e^{-2s}}}\left(\frac{2yze^{-s}}{1-e^{-2s}}\right)^{-\alpha}I_\alpha
\left(\frac{2yze^{-s}}{1-e^{-2s}}\right)\right]\nonumber\\
&=&2^{\alpha +1}(yz)^{\alpha
+1/2}e^{-\frac{1}{2}(y^2+z^2)\frac{1+e^{-2s}}{1-e^{-2s}}}\left\{
\left[-(\alpha
+1)\frac{1+e^{-2s}}{1-e^{-2s}}+2(y^2+z^2)\left(\frac{e^{-s}}{1-e^{-2s}}\right)^2\right]\right.\\
&\times& \left(\frac{e^{-s}}{1-e^{-2s}}\right)^{\alpha
+1}\left(\frac{2yze^{-s}}{1-e^{-2s}}\right)^{-\alpha}I_\alpha
\left(\frac{2yze^{-s}}{1-e^{-2s}}\right)\nonumber\\
&-&\left.(2yz)^2\left(\frac{e^{-s}}{1-e^{-2s}}\right)^{\alpha
+3}\frac{1+e^{-2s}}{1-e^{-2s}}\left(\frac{2yze^{-s}}{1-e^{-2s}}\right)^{-(\alpha
+1)}I_{\alpha
+1}\left(\frac{2yze^{-s}}{1-e^{-2s}}\right)\right\}.\nonumber
\end{eqnarray}
From (I1) and (P2) one has, for every $s,y,z\in (0,\infty )$ such
that $\frac{e^{-s}yz}{1-e^{-2s}}\leq 1$,
\begin{eqnarray}\label{modWs1}
\left|\frac{\partial }{\partial s}W_s^\alpha
(y,z)\right|&\leq&C(yz)^{\alpha
+1/2}e^{-\frac{1}{2}(y^2+z^2)\frac{1+e^{-2s}}{1-e^{-2s}}}\\
&\times&\left\{\left[\frac{1}{1-e^{-2s}}+(y^2+z^2)\left(\frac{1}{1-e^{-2s}}\right)^2\right]
\left(\frac{e^{-s}}{1-e^{-2s}}\right)^{\alpha
+1}\right.\nonumber\\
&+&\left.(yz)^2\left(\frac{e^{-s}}{1-e^{-2s}}\right)^{\alpha
+3}\frac{1}{1-e^{-2s}}\right\}\nonumber\\
&\leq&C(yz)^{\alpha +1/2}e^{-\frac{y^2+z^2}{8s}}\frac{e^{-(\alpha
+1)s}}{(1-e^{-2s})^{\alpha +2}}.\nonumber
\end{eqnarray}
Moreover, by using (I2) for $n=0$, it follows that, for every $s,y,z\in
(0,\infty )$ such that $\frac{e^{-s}yz}{1-e^{-2s}}\geq 1$,
\begin{eqnarray}\label{modWs2}
\hspace{1cm}\left|\frac{\partial }{\partial s}W_s^\alpha
(y,z)\right|&\leq&C(yz)^{\alpha
+1/2}e^{-\frac{1}{2}(y^2+z^2)\frac{1+e^{-2s}}{1-e^{-2s}}+2yz\frac{e^{-s}}{1-e^{-2s}}}\\
&\times&\left\{(yz)^{-\alpha
-1/2}\left(\frac{e^{-s}}{1-e^{-2s}}\right)^{1/2}\left[\frac{1}{1-e^{-2s}}+(y^2+z^2)
\left(\frac{e^{-s}}{1-e^{-2s}}\right)^2\right] \right.\nonumber\\
&+&\left.(yz)^{-\alpha
+1/2}\left(\frac{e^{-s}}{1-e^{-2s}}\right)^{3/2
}\left(\frac{1}{1-e^{-2s}}\right)\right\}\nonumber
\end{eqnarray}
\begin{eqnarray*}
&\leq&Ce^{-\frac{(z-e^{-s}y)^2+(y-e^{-s}z)^2}{2(1-e^{-2s})}}\left[\frac{e^{-s/2}}{(1-e^{-2s})^{3/2}}
+(y^2+z^2)\frac{e^{-5s/2}}{(1-e^{-2s})^{5/2}}+yz\frac{e^{-3s/2}}{(1-e^{-2s})^{5/2}}\right]\nonumber\\
&\leq&Ce^{-\frac{(z-e^{-s}y)^2+(y-e^{-s}z)^2}{2(1-e^{-2s})}}\frac{e^{-s/2}}{(1-e^{-2s})^{3/2}}\left[
1+(y^2+z^2)\frac{e^{-s}}{1-e^{-2s}}\right]\nonumber\\
&\leq&Ce^{-\frac{(z-e^{-s}y)^2+(y-e^{-s}z)^2}{2(1-e^{-2s})}}\frac{e^{-3s/2}}{(1-e^{-2s})^{5/2}}(y^2+z^2).
\end{eqnarray*}

Also, (I2), for $n=1$ and $n=2$, allows us to write, for
every $s,y,z\in (0,\infty )$ and $\frac{e^{-s}yz}{1-e^{-2s}}\geq
1$,
\begin{eqnarray*}
\frac{\partial }{\partial s}W_s^\alpha (y,z)&=&-
\frac{1}{\sqrt{\pi
}}e^{-\frac{1}{2}(y^2+z^2)\frac{1+e^{-2s}}{1-e^{-2s}}+2yz\frac{e^{-s}}{1-e^{-2s}}}
\frac{e^{-s/2}}{(1-e^{-2s})^{3/2}}\\
&\hspace{-4cm}\times &\hspace{-2cm}\left\{\left[(\alpha +1)
(1+e^{-2s})-\frac{2(y^2+z^2)e^{-2s}}{1-e^{-2s}}\right]
\left(1-[\alpha ,1]\frac{1-e^{-2s}}{4yze^{-s}}+O\left(\left(\frac{1-e^{-2s}}{yze^{-s}}\right)^2\right)\right)\right.\\
&\hspace{-4cm}+&\hspace{-2cm}\left.
2yz\frac{e^{-s}(1+e^{-2s})}{1-e^{-2s}}\left( 1-[\alpha
+1,1]\frac{1-e^{-2s}}{4yze^{-s}}+[\alpha
+1,2]\left(\frac{1-e^{-2s}}{4yze^{-s}}\right)^2+O\left(
\left(\frac{1-e^{-2s}}{yze^{-s}}\right)^3\right)\right)\right\}\\
&\hspace{-4cm}=&\hspace{-2cm}\frac{\partial}{\partial
s}W_s(y,z)-\frac{1}{\sqrt{\pi
}}e^{-\frac{(z-e^{-s}y)^2+(y-e^{-s}z)^2}{2(1-e^{-2s})}}\frac{e^{-s/2}}{(1-e^{-2s})^{3/2}}\\
&\hspace{-4cm}\times&\hspace{-2cm} \left\{\Big(\alpha
+\frac{1}{2}\Big)(1+e^{-2s})-[\alpha ,1](\alpha
+1)\frac{(1+e^{-2s})(1-e^{-2s})}{4yze^{-s}}+[\alpha
,1]\frac{(y^2+z^2)e^{-s}}{2yz}\right.\\
&\hspace{-4cm}-&\hspace{-2cm}\left.\frac{[\alpha
+1,1]}{2}(1+e^{-2s})+\frac{[\alpha
+1,2]}{2}\frac{(1+e^{-2s})(1-e^{-2s})}{4yze^{-s}}+
O\left(\frac{1-e^{-2s}}{yze^{-s}}\right)\right\}\\
&\hspace{-4cm}=&\hspace{-2cm}\frac{\partial}{\partial
s}W_s(y,z)-\frac{1}{\sqrt{\pi
}}e^{-\frac{(z-e^{-s}y)^2+(y-e^{-s}z)^2}{2(1-e^{-2s})}}\frac{e^{-s/2}}{(1-e^{-2s})^{3/2}}\\
&\hspace{-4cm}\times&\hspace{-2cm} \left\{-\frac{[\alpha ,1]}{2}(1-e^{-s})^2+O\left(\frac{(y-z)^2e^{-s}}{yz}\right)+
O\left(\frac{1-e^{-2s}}{yze^{-s}}\right)\right\}.
\end{eqnarray*}

Then, by using again (P1)(b) and (P2) we get, for every $s,y,z\in (0,\infty )$ such that
$\frac{e^{-s}yz}{1-e^{-2s}}\geq 1$,
\begin{eqnarray}\label{Ad}
\left|\frac{\partial }{\partial s}W_s^\alpha
(y,z)-\frac{\partial}{\partial s}W_s(y,z)\right| &&\nonumber\\
&\hspace{-4cm}\leq&\hspace{-2cm}
Ce^{-\frac{(y-z)^2(1+e^{-2s})}{2(1-e^{-2s})}-\frac{yz(1-e^{-s})}{2}}\left\{(1-e^{-s})^{1/2}
+\frac{(y-z)^2e^{-3s/2}}{yz(1-e^{-2s})^{3/2}}+\frac{e^{s/2}}{yz(1-e^{-s})^{1/2}}\right\}\nonumber\\
&\hspace{-4cm}\leq
&\hspace{-2cm}Ce^{-\frac{(y-z)^2}{2s}}
\frac{e^{s/2}}{yz(1-e^{-s})^{1/2}}.
\end{eqnarray}

We now study operators $A_1$ and $A_2$. Let $M$ and $N$ be the functions defined by
$$
M(x,z)=\left\{\int_{\Gamma _+(x)}\left|\left(s\frac{\partial
}{\partial s}W_s
(y,z)\right)_{|s=t^2}\right|^q\frac{dtdy}{t^2}\right\}^{1/q},\quad x,z\in (0,\infty ).
$$
and
$$
N(x,z)=\left\{\int_{\Gamma _+(x)}\left|\left(s\frac{\partial
}{\partial s}W_s^\alpha
(y,z)\right)_{|s=t^2}\right|^q\frac{dtdy}{t^2}\right\}^{1/q},\quad x,z\in (0,\infty ),
$$

By using (\ref{CZ1}) one gets
$$
|M(x,z)|\leq \frac{C}{|x-z|}\leq C\left\{\begin{array}{ll}
                        \displaystyle \frac{1}{x},&0<z<x/2,\\
                        &\\
                        \displaystyle \frac{1}{z},&0<2x<z,
                     \end{array}
\right.
$$
and that,
$$
|M(x,-z)|\leq \frac{C}{x+z},\quad x,z\in (0,\infty ).
$$

We also claim that
\begin{equation}\label{AcotN}
|N(x,z)|\leq C\left\{\begin{array}{ll}
                        \displaystyle \frac{1}{x},&0<z<x/2,\\[0.5cm]
                        \displaystyle \frac{1}{z},&0<2x<z.
                     \end{array}
\right.
\end{equation}
Thus, we can write
$$
A_1(|f|)(x)+A_2(|f|)(x)\leq C\left(\frac{1}{x}\int_0^x|f(z)|dz+\int_x^\infty \frac{|f(z)|}{z}dz\right),\quad x\in (0,\infty ).
$$
Wellknown properties of Hardy operators (\cite[p. 20]{Zyg})
allow us to establish that $A_1$ and $A_2$ are bounded operators
from $L^p(0,\infty )$ into itself, for every $1<p<\infty$,
and from $L^1(0,\infty )$ into $L^{1,\infty }(0,\infty
)$.

Let us establish (\ref{AcotN}). Denote by $L(z)$ and $R(z)$, $z\in
(0,\infty )$, the sets
\begin{equation}\label{L}
L(z)=\left\{(y,t)\in (0,\infty )\times (0,\infty
):\frac{yze^{-t^2}}{1-e^{-2t^2}}\leq 1\right\},
\end{equation}
and
\begin{equation}\label{R}
R(z)=\left\{(y,t)\in (0,\infty )\times (0,\infty ):\frac{yze^{-t^2}}{1-e^{-2t^2}}\geq 1\right\}.
\end{equation}

We have that $N(x,z)\leq N_1(x,z) +N_2(x,z)$, $x,z\in (0,\infty )$, where
$$
N_1(x,z)=\left\{\int_{\Gamma _+(x)\cap L(z)}\left|\left(s\frac{\partial
}{\partial s}W_s^\alpha
(y,z)\right)_{|s=t^2}\right|^q\frac{dtdy}{t^2}\right\}^{1/q},\quad x,z\in
(0,\infty ),
$$
and
$$
N_2(x,z)=\left\{\int_{\Gamma _+(x)\cap R(z)}\left|\left(s\frac{\partial
}{\partial s}W_s^\alpha
(y,z)\right)_{|s=t^2}\right|^q\frac{dtdy}{t^2}\right\}^{1/q} ,\quad
x,z\in (0,\infty ).
$$

From (\ref{modWs1}) and using (P3) and (P4) one gets
\begin{eqnarray}\label{N1}
N_1(x,z)&\leq&Cz^{\alpha +1/2}\left\{\int_0^\infty \int_{|x-y|}^\infty y^{q(\alpha
+1/2)}e^{-\frac{q(y^2+z^2)}{8t^2}}\frac{e^{-q(\alpha +1)t^2}t^{2q-2}}{(1-e^{-2t^2})^{q(\alpha
+2)}}dtdy\right\}^{1/q}\nonumber\\
&\leq& Cz^{\alpha +1/2}\left\{\int_0^\infty \int_{|x-y|}^\infty
\frac{y^{q(\alpha +1/2)}}{(t^2+y^2+z^2)^{q(\alpha +1)
+1}}\left(\frac{t^2e^{-\frac{\alpha +1}{\alpha
+2}t^2}}{1-e^{-2t^2}}\right)^{q(\alpha
+2)}dtdy\right\}^{1/q}\nonumber\\
&\leq& Cz^{\alpha +1/2}\left\{\int_0^\infty \int_{|x-y|}^\infty
\frac{1}{(t+y+z)^{q(\alpha +3/2)+2}}dtdy\right\}^{1/q}\nonumber\\
&\leq &Cz^{\alpha +1/2}\left\{\left(\int_0^x+\int_x^\infty \right)
\frac{1}{(|x-y|+y+z)^{q(\alpha +3/2)+1}}dy\right\}^{1/q}\nonumber\\
&\leq& C\frac{z^{\alpha +1/2}}{(x+z)^{\alpha +3/2}} \leq
C\left\{\begin{array}{ll}
            \displaystyle \frac{z^{\alpha +1/2}}{x^{\alpha +3/2}},
            &\displaystyle 0<z<x,\\[0.3cm]
            \displaystyle \frac{1}{z},&0<x<z .
        \end{array}
\right.
\end{eqnarray}

On the other hand, let $M_{x,z}=\max\{x,z\}$ and $m_{x,z}=\min
\{x,z\}$. By taking into account properties (P1)(c), (P2), (P3) and (P4) we obtain,
from (\ref{modWs2})
\begin{eqnarray}\label{N2}
N_2(x,z)&\leq&C\left\{\int_0^\infty \int_{|x-y|}^\infty
e^{-\frac{q(y-z)^2}{4t^2}}\frac{e^{-3qt^2/2}t^{2q-2}}{(1-e^{-2t^2})^{5q/2}}(y^2+z^2)^q
dtdy\right\}^{1/q}\nonumber\\
&\leq &C\left\{\int_0^\infty \int_{|x-y|}^\infty
\frac{(y^2+z^2)^q}{(t^2+|y-z|^2)^{3q/2+1}}dtdy\right\}^{1/q}\leq
C\left\{\int_0^\infty
\frac{(y+z)^{2q}}{(|x-y|+|y-z|)^{3q+1}}dy\right\}^{1/q}\nonumber\\
&\leq
&C\left(M_{x,z}^2\left\{\int_0^{m_{x,z}}\frac{1}{(x+z-2y)^{3q+1}}dy\right\}^{1/q}
+M_{x,z}^2\left\{\int_{m_{x,z}}^{M_{x,z}}
\frac{1}{|x-z|^{3q+1}}dy\right\}^{1/q}\right.\nonumber\\
&+&\left.\frac{M_{x,z}^2}{|x-z|^2}\left\{\int_{M_{x,z}}^\infty \frac{1}{(2y-x-z)^{q+1}}dy\right\}^{1/q}\right)\nonumber\\
&&\leq C\frac{M_{x,z}^2}{|x-z|^3}\leq C\left\{\begin{array}{ll}
                \displaystyle \frac{1}{x},&\displaystyle 0<z<\frac{x}{2},\\[0.5cm]
                \displaystyle \frac{1}{z},&0<2x<z.
                \end{array}
\right.
\end{eqnarray}

\noindent In the fourth inequality we have used that, for each $x,z\in
(0,\infty )$,  $h_{x,z}(y)=\frac{y+z}{2y-x-z}$, is a decreasing
function on $(0,\infty )$. Hence, $h_{x,z}(y)\leq
h_{x,z}(M_{x,z})\leq 2\frac{M_{x,z}}{|x-z|}$, when $y\geq
M_{x,z}$.

Estimations (\ref{N1}) and (\ref{N2}) lead to (\ref{AcotN}).

Finally we study the operator $A_3$. We need to estimate the function
$$
G(x,z)=\Biggl\{\int_{\Gamma_{+}(x)}\biggl|\left[s\frac{\partial}{\partial
s}(W_{s}^{\alpha}(y,z)-W_{s}(y,z))\right]_{|s=t^{2}}\biggr|^q
\frac{dtdy}{t^{2}}\Biggr\}^{1/q},\quad 0<\frac{x}{2}<z<2x.
$$
By using again the sets $L(z)$ and $R(z)$, $z\in (0,\infty )$ (see
(\ref{L}) and (\ref{R})) we write
\begin{eqnarray*}
G(x,z)&\leq &C\left[\left\{\int_{\Gamma _+(x)\cap
L(z)}\left|\left[s\frac{\partial }{\partial s}W_s^\alpha
(y,z)\right]_{|s=t^2}\right|^q\frac{dtdy}{t^2}\right\}^{1/q}\right.\\
&+&\left\{\int_{\Gamma _+(x)\cap L(z)}\left|\left[s\frac{\partial
}{\partial s}W_s
(y,z)\right]_{|s=t^2}\right|^q\frac{dtdy}{t^2}\right\}^{1/q}\\
&+&\left.\left\{\int_{\Gamma _+(x)\cap R(z)}\left|\left[s\frac{\partial
}{\partial s}(W_s^\alpha
(y,z)-W_s(y,z))\right]_{|s=t^2}\right|^q\frac{dtdy}{t^2}\right\}^{1/q}\right]\\
&=&C[N_1(x,z)+G_2(x,z)+G_3(x,z)],\quad 0<\frac{x}{2}<z<2x.
\end{eqnarray*}

From (\ref{N1}) we have that
\begin{equation}\label{N1b}
N_1(x,z)\leq \frac{C}{z},\quad 0<\frac{x}{2}<z<2x.
\end{equation}

On the other hand, by (\ref{acotpartials}) and (P2) it follows that, when
$(y,s)\in L(z)$,
\begin{equation}\label{dWsacot}
\left|\frac{\partial}{\partial s}W_s(y,z)\right|\leq
Ce^{-(y^2+z^2)\frac{1+e^{-2s}}{8(1-e^{-2s})}}\frac{e^{-s/2}}{(1-e^{-2s})^{3/2}}\leq
Ce^{-\frac{y^2+z^2}{8s}}\frac{e^{-s/2}}{(1-e^{-2s})^{3/2}}.
\end{equation}
We note that the right side of (\ref{dWsacot}) coincides with the
right hand side of (\ref{modWs1}) when $\alpha =-1/2$. Hence by
proceeding as in the proof of (\ref{N1}) we conclude that
\begin{equation}\label{G2}
G_2(x,z)\leq \frac{C}{z},\quad 0<\frac{x}{2}<z<2x.
\end{equation}

Finally, considering (\ref{Ad}) and again (P3) and (P4), we get, when $0<\frac{x}{2}<z<2x$,
\begin{eqnarray*}
G_3(x,z)&\leq&C\left\{\int_{\Gamma _+(x)\cap R(z)}\left(\frac{e^{t^2/2}e^{-\frac{(y-z)^2}{2t^2}}}
{zy(1-e^{-t^2})^{1/2}}\right)^qt^{2q-2}dtdy\right\}^{1/q}\nonumber\\
&\leq&C\left\{\int_0^{\frac{z}{2}}\int_{|x-y|}^\infty
e^{-\frac{q(y-z)^2}{2t^2}}\left(\frac{e^{-t^2/2}}
{(1-e^{-t^2})^{3/2}}\right)^qt^{2q-2}dtdy\right\}^{1/q}\nonumber\\
&+&\left.\int_{\frac{z}{2}}^\infty \int_{|x-y|}^\infty
e^{-\frac{q(y-z)^2}{2t^2}}\left(\frac{e^{-t^2/4}}
{(zy)^{1/4}(1-e^{-t^2})^{5/4}}\right)^qt^{2q-2}dtdy\right\}^{1/q}\nonumber\\
&\leq&C\left\{\int_0^{\frac{z}{2}}\int_{|x-y|}^\infty
\frac{1}{(t^2+(z-y)^2)^{q/2+1}}dtdy+\frac{1}{z^{q/2}}\int_{\frac{z}{2}}^\infty
\int_{|x-y|}^\infty \frac{1}{(t^2+(z-y)^2)^{q/4+1}}dtdy\right\}^{1/q}\nonumber\\
&\leq&C\left\{\int_0^{\frac{z}{2}}\frac{1}{(|x-y|+|z-y|)^{q+1}}dy+\frac{1}{z^{q/2}}\int_{\frac{z}{2}}^\infty
\frac{1}{(|x-y|+|z-y|)^{q/2+1}}dy\right\}^{1/q}\nonumber\\
&\leq&C\left\{\int_0^{\frac{z}{2}}\frac{1}{(x+z-2y)^{q+1}}dy+\frac{1}{z^{q/2}}\int_0^\infty
\frac{1}{(|x-y|+|z-y|)^{q/2+1}}dy\right\}^{1/q}\nonumber\\
&\leq &C\left\{\frac{1}{z^q}+\frac{1}{z^{q/2}|x-z|^{q/2}}\right\}^{1/q}.
\end{eqnarray*}
The second term in the last inequality can be obtained in the same way as in the proof of (\ref{CZ1}).
Then, we have
\begin{equation}\label{G3}
G_3(x,z)\leq \frac{C}{z}\left(1+\sqrt{\frac{z}{|x-z|}}\right),\quad 0<\frac{x}{2}<z<2x.
\end{equation}
and by (\ref{N1b}), (\ref{G2}) and (\ref{G3}) we conclude that
$$
G(x,z)\leq \frac{C}{z}\left(1+\sqrt{\frac{z}{|z-x|}}\right),\quad
0<\frac{x}{2}<z<2x.
$$

The operator $\mathfrak{B}$ defined by
$$
\mathfrak{B}(f)(x)=\int_{\frac{x}{2}}^{2x}\frac{1}{z}\left(1+\sqrt{\frac{z}{|z-x|}}\right)f(z)dz,\quad
x\in (0,\infty ),
$$
is bounded from $L^p(0,\infty )$ into itself for every $1\leq
p\leq \infty $. Indeed, note that
$$
\int_{\frac{x}{2}}^{2x}\frac{1}{z}\left(1+\sqrt{\frac{z}{|z-x|}}\right)dz=
\int_{\frac{1}{2}}^{2}\frac{1}{u}\left(1+\sqrt{\frac{u}{|1-u|}}\right)du>0,\quad
x\in (0,\infty ).
$$
Jensen's inequality allows us to show the boundedness on
$L^p(0,\infty )$, $1\leq p\leq \infty$, of the operator
$\mathcal{B}$. Hence the operator $A_3$ is bounded from
$L^p(0,\infty)$ into itself, for every $1\leq p\leq \infty $. The
proof is thus completed.

\section{Proof of Proposition \ref{reverse}}
In order to prove Proposition \ref{reverse} we first establish the following property.
\begin{propo}\label{polarizacion}
Let $1<p<\infty$. For every $f\in L^p(0,\infty )$ and $h\in L^{p'}(0,\infty )$, one has
$$
\int_0^\infty f(x)h(x)dx=8\int_0^\infty\int_{\Gamma _+(x)}s\frac{\partial }{\partial s}W_s^\alpha (f)(y)_{|s=t^2}
s\frac{\partial }{\partial s}W_s^\alpha (h)(y)_{|s=t^2}\frac{dtdy}{t|J_t(y)|}dx.
$$
where $J_t(y)=\{x\in (0,\infty ): |x-y|<t\}$, $y,t>0$. Here $p'$ denotes, as usual, the exponent conjugate of $p$.
\end{propo}
\begin{proof}
Let us consider the bilinear mappings $T_1$ and $T_2$ defined on $L^p(0,\infty )\times L^{p'}(0,\infty )$ by
$$
T_1(f,h)=\int_0^\infty f(x)h(x)dx,
$$
and
$$
T_2(f,h)=8\int_0^\infty\int_{\Gamma _+(x)}s\frac{\partial }{\partial s}W_s^\alpha (f)(y)_{|s=t^2}
s\frac{\partial }{\partial s}W_s^\alpha (h)(y)_{|s=t^2}\frac{dtdy}{t|J_t(y)|}dx.
$$

It is clear, by using the H\"older inequality, that $T_1$ is a continuous bilinear functional. Also, since $t|J_t(y)|\geq t^2$, $t,y>0$, the H\"older inequality and Proposition \ref{Laguerre} lead to
\begin{eqnarray*}
|T_2(f,g)|&\leq&8\int_0^\infty\left(\int_{\Gamma _+(x)}\left|s\frac{\partial }{\partial s}W_s^\alpha (f)(y)_{|s=t^2}
\right|^2\frac{dtdy}{t|J_t(y)|}\right)^{1/2}\\
&\times&\left(\int_{\Gamma _+(x)}\left|s\frac{\partial }{\partial s}W_s^\alpha (h)(y)_{|s=t^2}\right|^2\frac{dtdy}{t|J_t(y)|}\right)^{1/2}dx\\
&\leq&8\int_0^\infty g_{\mathbb{W}^\alpha}^2(f)(x)g_{\mathbb{W}^\alpha }^2(h)(x)dx\\
&\leq&8||g_{\mathbb{W}^\alpha}^2(f)||_{L^p(0,\infty )}||g_{\mathbb{W}^\alpha }^2(h)||_{L^{p'}(0,\infty )}\leq C||f||_{L^p(0,\infty )}||h||_{L^{p'}(0,\infty )}.
\end{eqnarray*}

Then, by taking into account that $\mbox{span }\{\varphi _n^\alpha\}_{n=1}^\infty$ is dense in $L^r(0,\infty )$, $1<r<\infty$ (see \cite[Lemma 4.3]{No1}), we will finish the proof when we show that $T_1(f,h)=T_2(f,h)$, for $f,g\in
\mbox{span }\{\varphi _n^\alpha\}_{n=1}^\infty$.

Let $f=\sum_{n=0}^\infty a_n\varphi _n^\alpha$ and $h=\sum _{n=0}^\infty b_n\varphi _n^\alpha$, where $a_n, b_n\in \mathbb{R}$, $n\in \mathbb{N}$, and $a_n, b_n\not=0$, only for a finite number of $n$. We can write
\begin{eqnarray*}
T_2(f,h)&=&8\sum_{n,m=0}^\infty a_nb_m(2n+\alpha +1)(2m+\alpha +1)\int_0^\infty \int_{\Gamma _+(x)}t^3e^{-2(n+m+\alpha +1)t^2}\varphi _n^\alpha (y)\varphi _m^\alpha (y)\frac{dtdy}{|J_t(y)|}dx\\
&=&8\sum_{n,m=0}^\infty a_nb_m(2n+\alpha +1)(2m+\alpha +1)\int_0^\infty t^3e^{-2(n+m+\alpha +1)t^2}\int_0^\infty \frac{\varphi _n^\alpha (y)\varphi _m^\alpha (y)}{|J_t(y)|}\int_{J_t(y)}dxdydt\\
&=&8\sum_{n=0}^\infty a_nb_n(2n+\alpha +1)^2\int_0^\infty t^3e^{-2(2n+\alpha +1)t^2}dt\\
&=&\sum_{n=0}^\infty a_nb_n=T_1(f,h).
\end{eqnarray*}
Thus the proof is finished.
\end{proof}

Let $1<p<\infty$ and $f\in L^p(0,\infty )$. By using duality and Proposition \ref{polarizacion} we have
\begin{eqnarray*}
||f||_{L^p(0,\infty )}&=&\sup_{h\in L^{p'}(0,\infty ), ||h||_{L^{p'}(0,\infty )}\leq 1}\left|\int_0^\infty f(x)h(x)dx\right|\\
&=&8\sup_{h\in L^{p'}(0,\infty ), ||h||_{L^{p'}(0,\infty )}\leq 1}\left|\int_0^\infty\int_{\Gamma _+(x)}\left[s\frac{\partial }{\partial s}W_s^\alpha (f)(y)\right]_{|s=t^2}
\left[s\frac{\partial }{\partial s}W_s^\alpha (h)(y)\right]_{|s=t^2}\frac{dtdy}{t|J_t(y)|}dx\right|.
\end{eqnarray*}

By taking into account that $t|J_t(y)|\geq t^2$, $y,t>0$, the H\"older inequality leads to
\begin{eqnarray*}
||f||_{L^p(0,\infty )}
&\leq&8\sup_{h\in L^{p'}(0,\infty ), ||h||_{L^{p'}(0,\infty )}\leq 1}\int_0^\infty
g_{\mathbb{W}^\alpha }^q(f)(x)g_{\mathbb{W}^\alpha}^{q'}(h)(x)dx\\
&\leq&8\sup_{h\in L^{p'}(0,\infty ), ||h||_{L^{p'}(0,\infty )}\leq 1}
||g_{\mathbb{W}^\alpha }^q(f)||_{L^p(0,\infty )}||g_{\mathbb{W}^\alpha}^{q'}(h)||_{L^{p'}(0,\infty )}.
\end{eqnarray*}

Since $q'\geq 2$, Proposition \ref{Laguerre} allows us finally to conclude that
$$
||f||_{L^p(0,\infty )}\leq C||g_{\mathbb{W}^\alpha}^q(f)||_{L^p(0,\infty )}.
$$

\section{Proof of Proposition \ref{H1}}
Let $f\in L^1(0,\infty )$. We denote by $f_{\rm o}$ the odd extension of $f$ to $\mathbb{R}$. According to \cite[Remark 2.4]{BDG}, $f\in H^1_{L_\alpha }(0,\infty )$ if, and only if, $f_{\rm o}\in H^1_{\mathcal H}(\mathbb{R})$. Moreover, $f_{\rm o}\in H^1_{\mathcal H}(\mathbb{R})$ when, and only when, $g^2_{\mathbb{W}}(f_{\rm o})\in L^1(\mathbb{R})$. Our objective is to see that $g_{\mathbb{W}}^2(f_{\rm o})\in L^1(\mathbb{R})$ if and only if $g_{\mathbb{W}^\alpha }^2(f)\in L^1(0,\infty )$.



According to the comments at the beginning of the proof of Proposition \ref{Laguerre}, $g_{\mathbb{W}}^2(f_{\rm o})\in L^1(\mathbb{R})$ if, and only if, $g_{\mathbb{W}}^{2,+}(f)$, defined as in (\ref{gq+}) with $q=2$, belongs to $ L^1(0,\infty )$.

By using (\ref{Ws}) and the Minkowski inequality, we get
\begin{eqnarray}\label{A2}
\left|g_{\mathbb{W}}^{2,+}(f)(x)-\left\{\int_{\Gamma _+(x)}\left|\left(s\frac{\partial }{\partial s}\int_{\frac{x}{2}}^{2x}W_s(y,z)f(z)dz\right)_{|s=t^2}\right|^2\frac{dtdy}{t^2}\right\}^{1/2}\right|&&\nonumber\\
&\hspace{-18cm}\leq&\hspace{-9cm}\int_{(0,\infty )\setminus (\frac{x}{2},2x)}\left\{\int_{\Gamma _+(x)}
\left|\left(s\frac{\partial }{\partial s}(W_s(y,z)-W_s(y,-z))\right)_{|s=t^2}\right|^2\frac{dtdy}{t^2}\right\}^{1/2}|f(z)|dz\nonumber\\
&\hspace{-18cm}-&\hspace{-9cm}\int_{\frac{x}{2}}^{2x}\left\{\int_{\Gamma _+(x)}\left|\left(s\frac{\partial }{\partial s}W_s(y,-z)\right)_{|s=t^2}\right|^2\frac{dtdy}{t^2}\right\}^{1/2}|f(z)|dz\nonumber\\
&\hspace{-18cm}=&\hspace{-9cm}G_1(|f|)(x)-G_2(|f|)(x),\quad x\in (0,\infty ).
\end{eqnarray}

By taking into account (\ref{CZ1}) with $q=2$ we obtain
$$
G_2(|f|)(x)\leq C\int_{\frac{x}{2}}^{2x}\frac{|f(z)|}{x+z}dz,\quad x\in (0,\infty ),
$$
and, consequently, $G_2$ defines a bounded operator from $L^1(0,\infty )$ into itself. On the other hand, according again to (\ref{CZ1}) with $q=2$, we have that
$$
\left\{\int_{\Gamma _+(x)}
\left|\left(s\frac{\partial }{\partial s}(W_s(y,z)-W_s(y,-z))\right)_{|s=t^2}\right|^2\frac{dtdy}{t^2}\right\}^{1/2}\leq \frac{C}{z},\quad 0<2x<z,
$$
and by (\ref{CZ2a}), with $q=2$, we get
\begin{eqnarray*}
\left\{\int_{\Gamma _+(x)}
\left|\left(s\frac{\partial }{\partial s}(W_s(y,z)-W_s(y,-z))\right)_{|s=t^2}\right|^2\frac{dtdy}{t^2}\right\}^{1/2}&&\\
&\hspace{-16cm}=&\hspace{-8cm}\left\{\int_{\Gamma _+(x)}
\left|\left(s\int_{-z}^z\frac{\partial ^2 }{\partial u\partial s}W_s(y,u)du\right)_{|s=t^2}\right|^2\frac{dtdy}{t^2}\right\}^{1/2}\\
&\hspace{-16cm}\leq&\hspace{-8cm}\int_{-z}^z\left\{\int_{\Gamma _+(x)}
\left|\left(s\frac{\partial ^2 }{\partial u\partial s}W_s(y,u)du\right)_{|s=t^2}\right|^2\frac{dtdy}{t^2}\right\}^{1/2}du\\
&\hspace{-16cm}\leq&\hspace{-8cm}C\int_{-z}^z\frac{1}{(x-u)^2}du\leq C\frac{z}{x^2},\quad 0<z<\frac{x}{2}.
\end{eqnarray*}
Then,
$$
G_1(|f|)(x)\leq C\left(\int_{2x}^\infty \frac{|f(z)|}{z}dz+\frac{1}{x^2}\int_0^{\frac{x}{2}}z|f(z)|dz\right),
\quad x\in (0,\infty ),
$$
and thus, $G_1$ is a bounded operator from $L^1(0,\infty )$ into itself.

From (\ref{A2}) we deduce that $g_\mathbb{W}^{2,+}(f)\in L^1(0,\infty )$ if and only if
$$
\left\{\int_{\Gamma _+(x)}
\left|\left(s\frac{\partial}{\partial s}\int_{\frac{x}{2}}^{2x}W_s(y,z)f(z)dz\right)_{|s=t^2}\right|^2\frac{dtdy}{t^2}\right\}^{1/2}\in L^1(0,\infty ).
$$

Note now that
\begin{eqnarray}\label{A2.5}
\left|g_{\mathbb{W}^\alpha}^2(f)(x)-\left\{\int_{\Gamma _+(x)}
\left|\left(s\frac{\partial}{\partial s}\int_{\frac{x}{2}}^{2x}W_s(y,z)f(z)dz\right)_{|s=t^2}\right|^2\frac{dtdy}{t^2}\right\}^{1/2}\right|&&\nonumber\\
&\hspace{-18cm}\leq &\hspace{-9cm}\int_{(0,\infty ) \setminus (\frac{x}{2},2x)}\left\{\int_{\Gamma _+(x)}
\left|\left(s\frac{\partial}{\partial s}W_s^\alpha (y,z)\right)_{|s=t^2}\right|^2\frac{dtdy}{t^2}\right\}^{1/2}|f(z)|dz\nonumber\\
&\hspace{-18cm}+&\hspace{-9cm}\int_{\frac{x}{2}}^{2x}\left\{\int_{\Gamma _+(x)}
\left|\left(s\frac{\partial }{\partial s}(W_s^\alpha (y,z)-W_s(y,z))\right)_{|s=t^2}\right|^2\frac{dtdy}{t^2}\right\}^{1/2}|f(z)|dz\nonumber\\
&\hspace{-18cm}=&\hspace{-9cm}\Lambda _1(|f|)(x)+\Lambda _2(|f|)(x),\quad x\in (0,\infty ).
\end{eqnarray}
As it was showed in the proof of Proposition \ref{Laguerre}, the operator $\Lambda _2$ is bounded from $L^1(0,\infty )$ into itself. (Note that $\Lambda _2(|f|)=A_3(|f|)$, with $q=2$).
On the other hand, from (\ref{AcotN}) for $q=2$, we have that
\begin{equation}\label{3.6q=2}
\left\{\int_{\Gamma _+(x)}
\left|\left(s\frac{\partial}{\partial s}W_s^\alpha (y,z)\right)_{|s=t^2}\right|^2\frac{dtdy}{t^2}\right\}^{1/2}\leq
\frac{C}{z}, \quad 0<2x<z.
\end{equation}
Moreover, let $L(z)$ and $R(z)$, $z\in (0,\infty )$, be as in (\ref{L}) and (\ref{R}), respectively, and consider
$$
N_1(x,z)=\left\{\int_{\Gamma _+(x)\cap L(z)}
\left|\left(s\frac{\partial }{\partial s}W_s^\alpha (y,z)\right)_{|s=t^2}\right|^2\frac{dtdy}{t^2}\right\}^{1/2}, \quad x,z\in (0,\infty ),
$$
and
$$
N_2(x,z)=\left\{\int_{\Gamma _+(x)\cap R(z)}
\left|\left(s\frac{\partial }{\partial s}W_s^\alpha (y,z)\right)_{|s=t^2}\right|^2\frac{dtdy}{t^2}\right\}^{1/2}, \quad x,z\in (0,\infty ).
$$
By (\ref{N1}), we get
\begin{equation}\label{A3}
N_1(x,z)\leq C\frac{z^{\alpha +1/2}}{x^{\alpha +3/2}},\quad 0<z<\frac{x}{2}.
\end{equation}

Also, by proceeding as in the proof of (\ref{N2}), we can obtain
\begin{eqnarray}\label{A4}
N_2(x,z)&\leq&C\left\{\int_{\Gamma _+(x)\cap R(z)} e^{-\frac{(y-z)^2}{2t^2}}\frac{e^{-3t^2}t^2}{(1-e^{-2t^2})^5}(y^2+z^2)^2dtdy\right\}^{1/2}\nonumber\\
&\leq &C\left\{\int_0^\infty \int_{|x-y|}^\infty e^{-\frac{(y-z)^2}{2t^2}}\frac{e^{-3t^2}t^2}{(1-e^{-2t^2})^5}\left(\frac{yze^{-t^2}}{1-e^{-2t^2}}\right)^{2\alpha +1}(y^2+z^2)^2dtdy\right\}^{1/2}\nonumber\\
&\leq &Cz^{\alpha +1/2}\left\{\int_0^\infty \int_{|x-y|}^\infty e^{-\frac{(y-z)^2}{2t^2}}\frac{e^{-(2\alpha +4)t^2}t^2}{(1-e^{-2t^2})^{2\alpha +6}}(y+z)^4y^{2\alpha +1}dtdy\right\}^{1/2}\nonumber\\
&\leq &Cz^{\alpha +1/2}\left\{\int_0^\infty \int_{|x-y|}^\infty \frac{(y+z)^4y^{2\alpha +1}}{(t^2+|y-z|^2)^{2\alpha +5}}dtdy\right\}^{1/2}\nonumber\\
&\leq &C\frac{z^{\alpha +1/2}x^{\alpha +5/2}}{|x-z|^{2\alpha +4}}\leq C\frac{z^{\alpha +1/2}}{x^{\alpha +3/2}},\quad 0<z<\frac{x}{2}.
\end{eqnarray}

From (\ref{3.6q=2}), (\ref{A3}) and (\ref{A4}) we deduce that
$$
\Lambda _1(|f|)(x)\leq C\left(\int_{2x}^\infty \frac{|f(z)|}{z}dz+\frac{1}{x^{\alpha +3/2}}\int_0^{x/2}z^{\alpha +1/2}|f(z)|dz\right),\quad x\in (0,\infty ).
$$
Hence, $\Lambda _1$ defines a bounded operator from $L^1(0,\infty )$ into itself.

We infer from (\ref{A2.5}) that $g_{\mathbb{W}^\alpha }^2(f)\in L^1(0,\infty )$ if and only if
$$
\left\{\int_{\Gamma _+(x)}
\left|\left(s\frac{\partial }{\partial s}\int_{\frac{x}{2}}^{2x}W_s(y,z)f(z)dz\right)_{|s=t^2}\right|^2\frac{dtdy}{t^2}\right\}^{1/2}\in L^1(0,\infty ).
$$
Thus, we conclude that $f\in H_{L_\alpha }^1(0,\infty )$ if and only if $g_{\mathbb{W}^\alpha }^2(f)\in L^1 (0,\infty )$. Moreover, the above estimations prove that the quantities $\|f\|_{H^1_{L_\alpha }(0,\infty )}$ and $\|f\|_{L^1(0,\infty )}+\|g_{\mathbb{W}^\alpha }^2(f)\|_{L^1(0,\infty )}$ are equivalent.



\begin{thebibliography}{10}
\bibitem{AtPi} L. Atanasi and M.A. Picardello, \textit{The Lusin area function and the local admisible convergence of harmonic functions on homogeneous trees}, Trans. Amer. Math. Soc. \textbf{360} (2008), 3327--3343.
\bibitem{BaMoo} R. Ba\~nuelos and Ch.N. Moore, \textit{Distribution function inequalities for the density of the area integral}, Ann. Inst. Fourier (Grenoble) \textbf{41} (1991), no. 1, 137--171.
\bibitem{BDG} J.J. Betancor, J. Dziuba\'nski and G. Garrig\'os, \textit{Riesz transform characterization of Hardy spaces associated with certain Laguerre expansions}, preprint 2009.
\bibitem{BFRST1} J.J. Betancor, J.C. Fari\~na, L. Rodr\'{\i}guez-Mesa, A. Sanabria and J.L. Torrea, \textit{Transference between Laguerre and Hermite settings}, J. Funct. Anal. \textbf{254} (2008), no. 3, 826--850.
\bibitem{BFRST2} J.J. Betancor, J.C. Fari\~na, L. Rodr\'{\i}guez-Mesa, A. Sanabria and J.L. Torrea, \textit{Lusin type and cotype for Laguerre g-functions}, to appear in Israel Journal of Mathematics.
\bibitem{BuGun} D.L. Burkholder and R.F. Gundy, \textit{Distribution function inequalities for the area integral}, Studia Math. \textbf{44} (1972), 527--544.
\bibitem{Cal} A.P. Calder\'on, \textit{On a theorem of Marcinkiewicz and Zygmund}, Trans. Amer. Math. Soc. \textbf{68} (1950), 55--61.
\bibitem{CH} A. Chicco Ruiz and E. Harboure, \textit{Weighted norm inequalities for the heat-diffusion Laguerre's semigroups}, Math. Z. \textbf{257} (2007), no. 2, 329--354.
\bibitem{Dahl} B.E. Dahlberg, \textit{Weighted norm inequalities for the Lusin area integral and the nontangential maximal functions for functions harmonic in a Lipschitz domain}, Studia Math. \textbf{67} (1967), no. 3, 297--314.
\bibitem{DKPV} B.E. Dalhberg, C.E. Kenig, J. Pipher, and G.C. Verchota, \textit{Area integral estimates for higher order elliptic equations and systems}, Ann. Inst. Fourier (Grenoble) \textbf{47} (1997), no. 5, 1425--1461.
\bibitem{Dz1} J. Dziuba\'nski, \textit{Hardy spaces for Laguerre expansions}, Constr. Approx. \textbf{27} (2008), 269--287.
\bibitem{Dz2} J. Dziuba\'nski, \textit{Atomic decomposition of Hardy spaces associated with certain Laguerre expansions}, J. Fourier Anal. Appl. \textbf{15} (2009), 129--152.
\bibitem{DZ1} J. Dziuba\'nski and J. Zienkiewicz, \textit{Hardy space $H^1$ associated to Schr\"odingger operator with potential satisfying reverse H\"older inequality}, Rev. Mat. Iberoamericana \textbf{15} (1999), 279--296.
\bibitem{DZ2} J. Dziuba\'nski and J. Zienkiewicz, \textit{$H^p$ spaces associated with Schr\"odinger operators with potential from reverse H\"older classes}, Colloq. Math. \textbf{98} (2003), 5--38.
\bibitem{DZ3} J. Dziuba\'nski and J. Zienkiewicz, \textit{Hardy spaces $H^1$ for Schr\"odinger operators with compactly supported potentials}, Ann. Mat. Pura Appl. (4) \textbf{184} (2005), 315--326.
\bibitem{FGS} E.B. Fabes, C.E. Gutierrez, and R. Scotto, \textit{Weak-type estimates for the Riesz transforms associated with the Gaussian measure}, Rev. Mat. Iberoamericana \textbf{10} (1994), no. 2, 229--281.
\bibitem{GMMST} J. Garc\'{\i}a-Cuerva, G. Mauceri, S. Meda, P. Sj\"ogren, and J.L. Torrea, \textit{Maximal operators for the holomorphic Ornstein-Uhlenbeck semigroup}, J. London Math. Soc. (2) \textbf{67} (2001), no. 1, 219--234.
\bibitem{GMST1} J. Garc\'{\i}a-Cuerva, G. Mauceri, P. Sj\"ogren, and J.L. Torrea, \textit{Higher-order Riesz operators for the Ornstein-Uhlenbeck semigroup}, Potential Anal. \textbf{10} (1999), no. 4, 379--407.
\bibitem{GMST2} J. Garc\'{\i}a-Cuerva, G. Mauceri, P. Sj\"ogren, and J.L. Torrea, \textit{Spectral multipliers for the Ornstein-Uhlenbeck semigroup}, J. Anal. Math. \textbf{78} (1999), 281--305.
\bibitem{Gas} G. Gasper, \textit{On the Littlewood-Paley and Lusin functions in higher dimensions}, Proc. Nat. Acad. Sci. U.S.A. \textbf{57} (1967), 25--28.
\bibitem{GLLNU} P. Graczyk, J.J. Loeb, I.A. L\'opez, A. Nowak, and W. Urbina, \textit{Higher order Riesz transforms, fractional derivatives and Sobolev spaces for Laguerre expansions}, J. Math. Pures Appl. (9) \textbf{84} (2005), no. 3, 375--405.
\bibitem{GuWh} R.F. Gundy and R.L. Wheeden, \textit{Weighted integral inequalities for the nontangential maximal function, Lusin area integral, and Walsh-Paley series}, Studia Math. \textbf{49} (1973/74), 107--124.
\bibitem{GIT} C.E. Gutierrez, A. Incognito, and J.L. Torrea, \textit{Riesz transforms, g-functions, and multipliers for the Laguerre semigroup}, Houston J. Math. \textbf{27} (2001), no. 3, 579--592.
\bibitem{HSTV} E. Harboure, C. Segovia, J.L. Torrea, and B. Viviani, \textit{Power weighted $L^p$-inequalities for Laguerre Riesz transforms}, Ark. Mat. \textbf{46} (2008), no. 2, 285--313.
\bibitem{HTV1} E. Harboure, J.L. Torrea, and B. Viviani, \textit{Vector-valued extensions of operators related to the Ornstein-Uhlenbeck semigroup}, J. Anal. Math. \textbf{91} (2003), 1--29.
\bibitem{HTV2} E. Harboure, J.L. Torrea, and B. Viviani, \textit{Riesz transforms for Laguerre expansions}, Indiana Univ. Math. J. \textbf{55} (2006), no. 3, 999--1014.
\bibitem{HMY} G. Hu, Y. Meng, and D. Yang, \textit{Estimates for Marcinkiewicz integrals in BMO and Campanato spaces}, Glasg. Math. J. \textbf{49} (2007), no. 2, 167--187.
\bibitem{HVP} T. Hyt\"onen, J. van Neerven, and P. Portal, \textit{Conical square function estimates in UMD Banach spaces and applications to $H^\infty$-functional calculi}, J. Anal. Math. \textbf{106} (2008), 317--351.
\bibitem{KoVa} A. Kor\'anyi and S. V\'agi, \textit{Singular integrals on homogeneous spaces and some problems of classical analysis}, Ann. Scuola Norm. Sup. Pisa (3) \textbf{25} (1971), 575--648.
\bibitem{Leb} N.N. Lebedev, \textit{Special functions and their applications}, Dover Publications Inc., New York, 1972.
\bibitem{MZ} J. Marcinkiewicz and A. Zygmund, \textit{A theorem of Lusin}, Duke Math. J. \textbf{4} (1938), no. 3, 473--485.
\bibitem{MST2} R. Mac\'{\i}as, C. Segovia, and J.L. Torrea, \textit{Heat-diffusion maximal operators for Laguerre semigroups with negative parameters}, J. Funct. Anal. \textbf{229} (2005), 300--316.
\bibitem{MTX} M.T. Mart\'{\i}nez, J.L. Torrea and Q. Xu, \textit{Vector-valued Littlewood-Paley-Stein theory for semigroups}, Adv. Math. \textbf{203} (2006), 430--475.
\bibitem{Mu1} B. Muckenhoupt, \textit{Hermite conjugate expansions}, Trans. Amer. Math. Soc. \textbf{139} (1969), 243--260.
\bibitem{Mu2} B. Muckenhoupt, \textit{Poisson integrals for Hermite and Laguerre expansions}, Trans. Amer. Math. Soc. \textbf{139} (1969), 231--242.
\bibitem{Mu5} B. Muckenhoupt, \textit{Conjugate functions for Laguerre expansions}, Trans. Amer. Math. Soc. \textbf{147} (1970), 403--418.
\bibitem{MuUchi} T. Murai and A. Uchiyama, \textit{Good $\lambda$ inequalities for the area integral and the nontangential maximal function}, Stud. Math. \textbf{83} (1985), 251--262.
\bibitem{No1} A. Nowak, \textit{Heat-diffusion and Poisson integrals for Laguerre and special Hermite expansions on weighted $L^p$ spaces}, Studia Math. \textbf{158} (2003), no. 3, 239--268.
\bibitem{No2} A. Nowak, \textit{On Riesz transforms for Laguerre expansions}, J. Funct. Anal. \textbf{215} (2004), no. 1, 217--240.
\bibitem{NS1} A. Nowak and K. Stempak, \textit{Riesz transforms and conjugacy for Laguerre function expansions of Hermite type}, J. Funct. Anal. \textbf{244} (2007), no. 2, 399--443.
\bibitem{NS2} A. Nowak and K. Stempak, \textit{Riesz transforms for multi-dimensional Laguerre function expansions}, Adv. Math. \textbf{215} (2007), no. 2, 642--678.
\bibitem{Pe} S. P\'erez, \textit{The local part and the strong type for operators related to the Gaussian measure}, J. Geom. Anal. \textbf{11} (2001), no. 3, 494--507.
\bibitem{PS} S. P\'erez and F. Soria, \textit{Operators associated with the Ornstein-Uhlenbeck semigroup}, J. London Math. Soc. \textbf{61} (2000), 857--871.
\bibitem{Seg} C. Segovia, \textit{On the area function of Lusin}, Studia Math. \textbf{33} (1969), 311--343.
\bibitem{Sj2} P. Sj\"ogren, \textit{On the maximal function for the Mehler kernel}, Harmonic Analysis (Cortona, 1982), Lecture Notes in Math., vol. 992, Springer, Berlin, 1983, pp. 73--82.
\bibitem{Sj1} P. Sj\"ogren, \textit{Operators associated with the Hermite semigroup - a survey}, J. Fourier Anal. Appl. \textbf{3} (1997), 813--823.
\bibitem{Spen} D.C. Spencer, \textit{A function-theoric identity}, Amer. J. Math. \textbf{65} (1943), 147--160.
\bibitem{Stein1} E.M. Stein, \textit{On the functions of Littlewood-Paley, Lusin and Marcinkiewicz}, Trans. Amer. Math. Soc. \textbf{88} (1958), 430--466.
\bibitem{Stein} E.M. Stein, \textit{Harmonic analysis: real-variable methods, orthogonality, and oscillatory integrals}, Princenton University Press, princenton, NJ, 1993.
\bibitem{St8} K. Stempak, \textit{Heat-diffusion and Poisson integrals for Laguerre expansions}, Tohoku Math. J. (2) \textbf{46} (1994), no. 1, 83--104.
\bibitem{StTo1} K. Stempak and J.L. Torrea, \textit{Poisson integrals and Riesz transforms for Hermite expansions with weights}, J. Funct. Anal. \textbf{202} (2003), no. 2, 443--472.
\bibitem{StTo2} K. Stempak and J.L. Torrea, \textit{On g-functions for Hermite function expansions}, Acta Math. Hungar. \textbf{109} (2005), no. 1-2, 99--125.
\bibitem{StTo4} K. Stempak and J.L. Torrea, \textit{BMO results for operators associated to Hermite expansions}, Illinois J. Math. \textbf{49} (2005), 1111--1131.
\bibitem{StTo3} K. Stempak and J.L. Torrea, \textit{Higher Riesz transforms and imaginary powers associated to the harmonic oscillator}, Acta Math. Hungar. \textbf{111} (2006), no. 1-2, 43--64.
\bibitem{Sz} G. Szeg\"o, \textit{Orthogonal polynomials}, fourth ed., American Mathematical Society, Providence, R.I., 1975.
\bibitem{Th2} S. Thangavelu, \textit{Lectures on Hermitte and Laguerre expansionos}, Mathematical Notes, vol. 42, Princenton University Press, Princenton, NJ, 1993.
\bibitem{Ur} W. Urbina, \textit{On singular integrals with respect to the Gaussian measure}, Ann. Scuola Norm. Sup. Pisa Cl. Sci. (4) \textbf{17} (1990), no. 4, 531--567.
\bibitem{Xu} Q. Xu, \textit{Littlewood-Paley theory for functions with values in uniformly convex spaces}, J. reine angew. Math. \textbf{504} (1998), 195--226.
\bibitem{Zyg} A. Zygmund, \textit{Trigonometric series}, Vol. I, Cambridge University Press, New York, 1959.
\end{thebibliography}

\end{document}